\documentclass{amsart}
\usepackage[utf8]{inputenc}
\usepackage[english]{babel}
\usepackage{graphicx}
\usepackage{array}
\usepackage{amsmath}
\usepackage{amsthm}
\usepackage{amsfonts} 
\usepackage[T1]{fontenc}
\usepackage{graphicx}
\usepackage[bottom=3cm, left=2.5cm, right=2.5cm, top=2.5cm]{geometry}
\usepackage{subfigure} % układ rysunków
\usepackage{color} % colory
\usepackage{cite} % bibliografia
\usepackage[status=draft]{fixme}
\fxusetheme{color}
\fxusetargetlayout{colorcb}
\fxuseenvlayout{color}

\newtheorem{lemat}{Lemma}
\newtheorem{tw}{Theorem}

\newtheorem{ass}{Assumption}
\newtheorem{defi}{Definition}
\newtheorem*{uwaga}{Remark}

\def\vc#1{\boldsymbol{#1}}
\def\ten#1{\boldsymbol{#1}}

\newcommand{\theep}{\theta^{\varepsilon}}
\newcommand{\theet}{\theta^{\eta}}

\newcommand{\dxdt}{\,{\rm{d}}x\,{\rm{d}}t}
\newcommand{\dx}{\,{\rm{d}}x}

\newcommand{\ds}{\,{\rm{d}}s}
\newcommand{\dt}{\,{\rm{d}}t}
\newcommand{\dz}{\,{\rm{d}}z}
\newcommand{\dtau}{\,{\rm{d}}\tau}
\newcommand{\dnu}{\,{\rm{d}}\nu}
\newcommand{\braket}[1]{(\!(#1)\!)}

\newcommand{\tenepp}{\ten{\varepsilon}^{\bf p}}

\newcommand{\tenepuk}{\ten{\varepsilon}(\vc{u}_{k})}

\numberwithin{equation}{section}

\def\div{\rm{div\,}}

\title{Thermo-visco-elasticity for models with growth conditions in Orlicz spaces}
\author[F. Z. Klawe]{Filip Z. Klawe}
\address{Institute of Applied Mathematics, University of Warsaw, ul. Banacha 2, 02-097 Warsaw, Poland}
\email{fzklawe@mimuw.edu.pl}

\begin{document}

\begin{abstract}

We study a quasi-static evolution of thermo-visco-elastic model. We act with external forces on non-homogeneous material body, which is a subject of our research. Such action may cause deformation of this body and may change its temperature. Mechanical part of the model contains two kinds of deformation: elastic and visco-elastic. Mechanical deformation is coupled with the temperature and they may influence each other. Since constitutive function on evolution of visco-elastic deformation depends on temperature, the visco-elastic properties of material also depend on temperature. We consider the thermodynamically complete model related to hardening rule with growth condition in generalized Orlicz spaces. We provide the proof of existence of solutions for such class of models.

\end{abstract}

\maketitle

\section{Introduction}

The objective of this paper is to show the existence of solution to special class of thermo-visco-elastic models. We consider reaction of material body treated by external forces and heat flux through the boundary. In the case of ideal elastic deformations, the body should return to its initial state after termination of external forces activity. However, if deformations are not elastic, i.e. there is a loss of potential energy, we deal with special kind of inelastic deformations. Potential energy lost during the process may be transformed into thermal energy. We focus on the visco-elastic type of deformations, which for instance may be observed in polymers. Both deformations are coupled in physical phenomena and they may be observed at the same time. Consequently, these two types of deformations appear in the models considered in this paper. Elastic deformation is reversible, whereas visco-elastic one irreversible.

The thermo-visco-elastic system of equations, as a consequence of balance of momentum and balance of energy, cf. \cite{GreenNaghdi, LandauLifshitz}, see  also \cite{GKSG}, captures displacement, temperature and visco-elastic strain. Since these two principles do not take into account the material properties of considered body, we may complement it by adding constitutive relations which complete missing information. The standard technique in the solid body deformation is to work with two constitutive relations. First one describes the dependency between stress and strains, i.e. this is an equation for the Cauchy stress tensor. Second one is a constitutive equation which is characterized by the evolution of visco-elastic strain tensor.

We assume that the body $\Omega \subset \mathbb{R}^3$ is an open bounded set with a $C^2$ boundary. Then quasi-static evolution problem is formulated by the following system of equations
\begin{equation}
\left\{
\begin{array}{rclr}
- \div \ten{T} &=& \vc{f} & \mbox{in } \Omega\times(0,T),
\\
\ten{T} &=& \ten{D}(\ten{\varepsilon}(\vc{u}) - \ten{\varepsilon}^{\bf p} ) & \mbox{in } \Omega\times(0,T),
\\
\ten{\varepsilon}^{\bf p}_t &=& \ten{G}(\theta,\ten{T}^d) & \mbox{in } \Omega\times(0,T),
\\
\theta_t - \Delta\theta &=& \ten{T}^d:\ten{G}(\theta,\ten{T}^d) & \mbox{in } \Omega\times(0,T).
\end{array}
\right.
\label{full_system_2}
\end{equation}
By the solution of this system we understand finding the displacement of material $\vc{u}:\Omega\times\mathbb{R}_+\rightarrow \mathbb{R}^3$, the temperature of material $\theta:\Omega\times\mathbb{R}_+\rightarrow\mathbb{R}_+$ and the visco-elastic strain tensor $\ten{\varepsilon}^{\bf p}:\Omega\times\mathbb{R}_+\rightarrow \mathcal{S}^3_d$. We denote by $\mathcal{S}^3$ the set of symmetric $3 \times 3$-matrices with real entries and by $\mathcal{S}^3_d$ a subset of $\mathcal{S}^3$ which contains traceless matrices. The function $\ten{T}:\Omega\times\mathbb{R}_+\rightarrow \mathcal{S}^3$ stays for the Cauchy stress tensor. By $\ten{I}$ we mean the identity matrix from $\mathcal{S}^3$, thus $\ten{T}^d$ is a deviatoric (traceless) part of the tensor $\ten{T}$, i.e. $\ten{T}^d=\ten{T}-\frac{1}{3}tr(\ten{T})\ten{I}$. Additionally, we denote by $\ten{\varepsilon}(\vc{u})$ the deformation tensor associated to $\vc{u}$, i.e. $\ten{\varepsilon}(\vc{u})=\frac{1}{2}(\nabla\vc{u} + \nabla^T\vc{u})$. 

The motivation for current paper is to extend results presented in \cite{GKSG_NH}, where we proved the existence of solution to Norton-Hoff model, i.e. the model with growth condition on visco-elastic strain tensor in Lebesgue spaces. Model with growth condition in the generalized Orlicz spaces is a natural extension of Norton-Hoff model as a next step to make an approximation of Prandtl-Reuss model. Use of generalized Orlicz spaces takes into consideration more rapid growth than in the case of growth condition in Lebesgue spaces. Furthermore, choice of generalized Orlicz space allow us to consider non-homogeneous materials. Since the $N$-function depends on the spatial variable $x$, different regions of $\Omega$  may have different growth condition. Consideration of non-homogeneous materials implies that the operator $\ten{D}$ may also depend on the spatial variable $x$. In previous papers, see \cite{GKSG,GKSG_NH} we considered only homogeneous materials. 

Studying mechanical problems in Orlicz spaces is not an isolated issue. In the case of visco-elastic deformation, the problem involving Orlicz spaces was considered in \cite{ChGw2007}, but only in the case of $N$-function independent on spatial variable $x$. In the case of $N$-function which depends on the spatial variable $x$ some accurate assumptions must be done. There are two possible ways to make it. Firstly, we may assume the regularity with respect to $x$, e.g. log-H\"{o}lder continuity in \cite{2Agnieszka, 3Agnieszka}, secondly upper and lower growth condition of an $N$-function with respect to the last variable can be considered, e.g. see \cite{Wroblewska20104136,GSG_TMNA,GSW,G1}. There are no results for thermo-visco-elastic problems without any upper and lower growth condition on $N$-function with respect to the last variable.

System of equations \eqref{full_system_2} is a mathematical simplification of more general model. We consider the quasi-static evolution with small displacement. It means that we omit acceleration term in momentum equation as a consequence of long-term character of external forces. Small displacement allows us to use the Hooke's law in the definition of Cauchy stress tensor \eqref{full_system_2}$_{(2)}$. Moreover, the material does not change its volume with the temperature, i.e. there is no thermal expansion of body, thus the Cauchy stress tensor does not depend on temperature.

System \eqref{full_system_2} may be completed by formulating the initial 
\begin{equation}
\left\{
\begin{array}{rcl}
\theta(x,0)&=&\theta_0(x), 
\\
\ten{\varepsilon}^{\bf p}(x,0)&=&\ten{\varepsilon}^{\bf p}_0(x),
\end{array}
\right.
\label{init_0}
\end{equation}
in $\Omega$ and boundary conditions
\begin{equation}
\left\{
\begin{array}{rcl}
\vc{u}&=&\vc{g}, \\
\frac{\partial \theta}{\partial \vc{n}}&=&g_{\theta},
\end{array}
\right.
\label{boun_0}
\end{equation}
on $\partial\Omega\times(0,T)$. We control the shape of $\Omega$ and the heat flux through the boundary.

To discuss two other equations and to formulate the statement of this paper, we need to use some definitions which are mentioned below for better readability of the paper. Let us begin with presenting the notion of generalized Orlicz spaces. For more general concept of Orlicz space we refer the reader to \cite{adams,musielak,raoren,odAnety}. We start with defining $N$-function.

\begin{defi}\label{df:Nfunction}
Let $\Omega$ be a bounded open domain in $\mathbb{R}^3$. A function $M:\Omega\times\mathcal{S}^3 \to \mathbb{R}_+$ is said to be $N$-function if it satisfies the following conditions:
\begin{itemize}
\item[1)] $M$ is a Carath\'{e}odory function (measurable with respect to $x$ and continuous with respect to $\ten{\xi}$) such that $M(x,\ten{\xi})=0$ if and only if $\ten{\xi}=\ten{0}$;
\item[2)] $M(x,\ten{\xi})=M(x,-\ten{\xi})$ a.e. in $\Omega$;
\item[3)] $M(x,\ten{\xi})$ is a convex function with respect to $\ten{\xi}$;
\item[4)] $\lim_{|\ten{\xi}|\to 0}M(x,\ten{\xi})/|\xi| =0$ for all $x\in\Omega$;
\item[5)] $\lim_{|\ten{\xi}|\to \infty}M(x,\ten{\xi})/|\xi| =\infty$ for all $x\in\Omega$;
\end{itemize}
\end{defi}

\begin{defi}
Function $M^*$ which is complementary to function $M$ is defined by 
%The complementary function $M^*$ to a function $M$ is defined by 
\begin{equation}
M^*(x,\ten{\eta}) = \sup_{\ten{\eta}\in\mathcal{S}^3}(\ten{\xi}:\ten{\eta} - M(x,\ten{\xi})),
\end{equation}
for $\ten{\eta}\in\mathcal{S}^3, x\in\Omega$.
\end{defi}

\begin{uwaga}
A complementary function $M^*$ to $N$-function $M$ is also an $N$-function.
\end{uwaga}

% \fxnote{koment}
Let us denote by $Q=\Omega \times (0,T)$. The generalized Orlicz %-Musielak 
class $\mathcal{L}_M(Q)$ is the set of all measurable function $\ten{\xi}:Q\to \mathcal{S}^3$ such that 
\begin{equation}
\int_Q M(x,\ten{\xi}(x,t)) \dxdt <\infty.
\end{equation}
The generalized Orlicz %-Musielak 
space $L_M(Q)$ can be defined as the smallest linear space containing $\mathcal{L}_M(Q)$. By $E_M(Q)$ we denote the closure of the set of bounded functions in $L_M$-norm. The generalized Orlicz space $L_M(Q)$ is a Banach space with respect to the Orlicz norm
\begin{equation}
\|\ten{\xi}\|_{O,M} = \sup \left\{\int_Q \ten{\xi}:\ten{\eta} \dxdt : \ \ten{\eta}\in L_{M^*}(Q), \ \int_Q M^*(x,\ten{\eta}) \dxdt \leq 1\right\},
\end{equation}
or equivalently with respect to Luxemburg norm
\begin{equation}
\|\ten{\xi}\|_{L,M} = \inf \left\{\lambda >0: \ \int_Q M\left(x,\frac{\ten{\xi}}{\lambda}\right) \dxdt \leq 1 \right\}.
\end{equation}

\begin{defi}
We say that an $N$-function $M$ satisfies $\Delta_2$-condition if for almost all $x\in\Omega$ and for all $\ten{\xi}\in\mathcal{S}^3$, there exists a constant $c$ and nonnegative integrable function $h:\Omega\to \mathbb{R}$ such that
\begin{equation}
M(x,2\ten{\xi})\leq c M(x,\ten{\xi}) + h(x).
\label{eq:def_delta2}
\end{equation}
\end{defi}

\begin{uwaga}
For every $M$ the following concluding holds
\begin{equation}
E_M(Q) \subseteq \mathcal{L}_M(Q) \subseteq L_M(Q).
\end{equation}
In particular, if $M$ satisfies the $\Delta_2$-condition, $E_M(Q) = L_M(Q)$. If $\Delta_2$-condition fails, we lose numerous properties of the space $L_M(Q)$ like separability, reflexivity and many others, cf. \cite{adams,musielak} and particular \cite{Minak} for generalized Orlicz spaces. 
\end{uwaga}

The space $L_{M^*}(Q)$ is the dual space of $E_M(Q)$. The functional
\begin{equation}
\rho (\ten{\xi}) = \int_Q M(x,\ten{\xi})\dxdt
\end{equation}
is a modular.%, see e.g. \cite{adams,musielak} for definition.

\begin{defi}
We say that a sequence $\{\ten{\xi}_i\}_{i=1}^{\infty}$ converges modularly to $\ten{\xi}$ in $L_M(Q)$ if there exists $\lambda>0$ such that
\begin{equation}
\int_Q M\left(x,\frac{\ten{\xi}_i-\ten{\xi}}{\lambda}\right)\dxdt \to 0,
\end{equation}
when $i$ is going to $\infty$. We use the notation $\ten{\xi}_i \xrightarrow{M}\ten{\xi}$ for modular convergence in $L_M(Q)$.
\end{defi}

In Appendix \ref{B} we present several lemmas related to Orlicz spaces. We use these lemmas to prove the existence of thermo-visco-elasticity model solution.

After these few definitions we may discuss the constitutive relations used to complement the system \eqref{full_system_2}. Relation between the Cauchy stress tensor and the strain tensor is defined by operator $\ten{D}:\mathcal{S}^3\rightarrow\mathcal{S}^3$, which is linear, positively definite and bounded. Moreover, $\ten{D}$ is a four-index matrix, i.e. $\ten{D}=\ten{D}(x)=\left\{d_{i,j,k,l}(x)\right\}_{i,j,k,l=1}^3$ and the following equalities  hold
\begin{equation}
d_{i,j,k,l}(x) = d_{j,i,k,l}(x),
\quad
d_{i,j,k,l}(x) = d_{i,j,l,k}(x)
\quad
\mbox{and}
\quad
d_{i,j,k,l}(x) = d_{k,l,i,j}(x),
\end{equation}
for all $i,j,k,l=1,2,3$ and for a.a. $x\in\Omega$. Additionally, function $d_{i,j,k,l}$ belongs to $W^{1,p}(\Omega)$ for each $i,j,k,l=1,2,3$  and for some $p>3$.

Second constitutive relation is an evolutionary equation for visco-elastic strain tensor. Function $\ten{G}:\mathbb{R}_+\times \mathcal{S}^3_d \to \mathcal{S}^3_d$ is a function of temperature and deviatoric part of Cauchy stress tensor. We discuss more precisely the concept of such a choice in \cite{GKSG_NH}. The properties of considered material imply the choice of specific function. Various different models were considered, e.g. Bodner-Partom model \cite{Bartczak, MANA:MANA5, MMA:MMA802}, Mr\'{o}z model \cite{GKSG, brokate, Homberg200455}, Norton-Hoff model \cite{GKSG_NH,CheAl}, Prandtl-Reuss model with linear kinematic hardening \cite{ChR}.

%We assume that function $\ten{G}(\cdot,\cdot)$ is monotone with respect to second variable $\ten{T}^d$ and function $\ten{G}(\cdot,\cdot)$ have a growth conditions in Orlicz space with respect to $\ten{T}^d$. 
\begin{ass}
The function $\ten{G}(\theta,\ten{T}^d)$ is continuous with respect to $\theta$ and $\ten{T}^d$ and satisfies the following conditions:
\begin{itemize}
\item[a)] $(\ten{G}(\theta,\ten{T}^d_1)-\ten{G}(\theta,\ten{T}^d_2)):(\ten{T}^d_1-\ten{T}^d_2) \geq 0$, for all $\ten{T}_1^d,\ten{T}_2^d \in \mathcal{S}^3_d$ and $\theta\in\mathbb{R}_+$;
\item[b)] $\ten{G}(\theta,\ten{T}^d):\ten{T}^d \geq c\left(  M(x,\ten{T}^d) + M^*(x,\ten{G}(\theta,\ten{T}^d))\right) $, where $\ten{T}^d\in\mathcal{S}^3_d$, $\theta\in\mathbb{R}_+$ and $c$ is a positive constant independent of temperature $\theta$.
\end{itemize}
\label{ass_G}
Moreover, we assume that the generalized Orlicz spaces fulfill:
\begin{itemize}
\item[1)] the following inequality holds
\begin{equation}
\int_Q M^*(x,\ten{A}(x,t))\dxdt \leq \int_Q |\ten{A}|^2 \dxdt \qquad \forall \ten{A} \in L_{M^*}(Q);
\end{equation}
\item[2)] $M^*$ satisfies the $\Delta_2$-condition.
\end{itemize}
%\label{ass:Mstar_bou_by_L2}
\end{ass}

Dealing with such assumption on function $\ten{G}(\cdot,\cdot)$ implies the displacement space. 

\begin{defi}{Bounded deformation space, see \cite{Stokes}}

Let us define the space of bounded deformation $BD_M(Q,\mathbb{R}^3)$
\begin{equation}
BD_{M}(Q,\mathbb{R}^3) = \left\{ \vc{u} \in L^1(\Omega,\mathbb{R}^3): \quad \ten{\varepsilon}(\vc{u}) \in L_{M}(\Omega,\mathcal{S}^3) \right\}.
\end{equation}
The space $BD_{M}(Q,\mathbb{R}^3)$ is a Banach space with a norm
\begin{equation}
\|\vc{u}\|_{BD_{M}(Q)} = \|\vc{u}\|_{L^1(Q)} + \|\ten{\varepsilon}(\vc{u})\|_{O,M}.
\end{equation}
According to [26,Theorem 1.1] there exists a unique continuous operator $\gamma_0$ from $BD_{M}(Q)$ onto $L^1((0,T)\times\partial\Omega)$ such that the generalized Green formula 
\begin{equation}
2\int_Q \phi \ten{\varepsilon}_{i,j}(\vc{u}) \dxdt
= 
-\int_Q (u_i\frac{\partial\phi}{\partial x_i} +u_j\frac{\partial\phi}{\partial x_j} )
+\int_0^T\int_{\partial\Omega} \phi (\gamma_0(u_i) n_j + \gamma_0(u_j) n_i) \mbox{d}\mathcal{H}^{n-1}\dt
\end{equation}
holds for every $\phi\in C^1(\overline{Q})$ and where $\vc{n}=(n_1,n_2,n_3)^T$ is an unit outward normal vector on $\partial\Omega$ and $\mathcal{H}^{n-1}$ is the $(n-1)$-Hausdorff measure. 

\label{df:sol_u}
\end{defi}

%{\color{red} Powinniśmy dodać warunek brzegowy!!!}
%For function $\vc{g}\in {\color{red}XXX} $ let us denote $BD_{M^*,\vc{g}}(Q,\mathbb{R}^3)=\{\vc{u}\in BD_{M^*}(Q,\mathbb{R}^3):\ \vc{u}=\vc{g} \mbox{ on }\partial\Omega\times(0,T)\}$.

Furthermore, we understand $\vc{v}\in BD_{M^*}(Q,\mathbb{R}^3) + L^{\infty}(0,T,W^{2,p}(\Omega,\mathbb{R}^3))$ in the following way: There exists a decomposition $\vc{v}= \vc{v}_1+ \vc{v}_2$, where $\vc{v}_1\in BD_{M^*}(Q,\mathbb{R}^3)$ and $\vc{v}_2\in L^{\infty}(0,T,W^{2,p}(\Omega,\mathbb{R}^3))$.

%{\color{red}
%Motivation for current paper is to extend the results presented in \cite{GKSG_NH}, where we prove the existence of solution to Norton-Hoff model. The growth condition in the generalized Orlicz spaces is a natural extension of Norton-Hoff model. Firstly, the usage of generalized Orlicz spaces takes into consideration more rapid growth than the Lebesgue spaces and hence better approximation of Prandtl-Reuss model.
%\fxnote{Czy tu może tak na pewno być? Czy nie ma prblemu np z wartościami własnymi?}
%Secondly, we may consider non-homogeneous material. In the case of generalized Orlicz spaces the $N$-function depends on spatial variable $x$. Thus, in different regions of $\Omega$ we may have different growth condition. Furthermore, operator $\ten{D}$ may also depend on the spatial variable.
%}

%{\color{red}
In the contrast to \cite{GKSG_NH} or \cite{Homberg200455} we use another approach to heat equation. By Assumption \ref{ass_G} we know that the right hand side function $\ten{G}(\theta,\ten{T}^d):\ten{T}^d$ is only integrable. Following Boccardo and Gallou\"{e}t, cf. \cite{Boccardo}, we proved in \cite{GKSG_NH} that the solution to the heat equation belongs to $L^q(0,T,W^{1,q}(\Omega))$ for all $q\in (1,\frac{5}{4})$. A weak point of this approach is lack of uniqueness. Hence, we change the approach and following Blanchard and Murat we prove the existence of a renormalised solution. The concept of renormalised solutions to parabolic equation with Dirichlet boundary condition was presented in \cite{Blanchard, BlanchardMurat}. In the Appendix \ref{A} we prove existence of a renormalised solution in case of Neumann boundary condition.

During the modelling of physical phenomena we should not forget about their physical properties. Losses of energy or admission of negative temperature causes that mathematical result has no physical interpretation. In the case of solid mechanics, the model should fulfill the principle of thermodynamics. Considered model is thermodynamically complete, i.e. that the principle of thermodynamics is fulfilled. In \cite{GKSG_NH,GKSG} we discussed conservation of energy, positivity of temperature and existence of entropy, which has a positive rate of production. Considering the quasi-static evolution causes that the energy of system consists of internal (thermal) and potential energy. Lack of acceleration term in balance of momentum implies that the kinetic energy of $\Omega$ fails.% is not taken into acount

%thermodynamically completeness: similar to \cite{GKSG,GKSG_NH}. Case with omitted thermal extension fulfills the thermodynamic laws: conservation of energy, positivity of temperature and positive rate of entropy production. 
%}

%{\color{red}Another physical properties}

%Idea wykorzystania warunków wzrostu w przestrzeniach orlicza, 

%{\color{red} Let us assume that $\Omega$ is non-homogeneous. The evolution of the visco-elastic strain tensor depend on the spatial coefficient $x$ (because of the $M$-function). Hence we should assume that the operator $\ten{D}$ depends on $x$. Then the following assumption is necessary in the calculations.
%}

%{\color{blue}KONIEC}

\begin{defi}{Weak-renormalised solution of the system \eqref{full_system_2} }

%Let $p\geq 2$.
The triple of functions $\vc{u}\in  BD_{M^*}(Q,\mathbb{R}^3) + L^{\infty}(0,T,W^{2,p}(\Omega,\mathbb{R}^3))$, $\ten{T}\in L^2(0,T,L^2(\Omega,\mathcal{S}^3))$ and $\theta\in  C([0,T],L^1(\Omega))$ such that for every $K\in\mathbb{N}$, $\mathcal{T}_K(\theta)\in L^2(0,T,W^{1,2}(\Omega))$ is a weak-renormalised solution of the system \eqref{full_system_2} when
\begin{equation}
\begin{split}
\int_0^T\int_{\Omega}\ten{T}:\nabla\vc{\varphi} \dxdt 
&= \int_0^T\int_{\Omega}\vc{f}\cdot \vc{\varphi} \dxdt ,
%\\
%\ten{T} = \ten{D}(\ten{\varepsilon} -\ten{\varepsilon}^{\bf p} ),
%\\
%- \int_0^T\int_{\Omega} \ten{\varepsilon}^{\bf p}:\ten{\Phi}_t  \dxdt -
%\int_{\Omega} \ten{\varepsilon}^{\bf p}_0(x):\ten{\Phi}(x,0) \dx &=
%\int_0^T\int_{\Omega} \ten{G}(\theta,\ten{T}^d):\ten{\Phi} \dxdt ,
\end{split}
\end{equation}
where 
\begin{equation}
\ten{T}=\ten{D}(\ten{\varepsilon}(\vc{u}) - \ten{\varepsilon}^{\bf p}),
\end{equation}
holds for every test function $\vc{\varphi}\in C^{\infty}([0,T],C^{\infty}_c(\Omega,\mathbb{R}^3))$ and 
\begin{equation}
\begin{split}
-\int_Q S(\theta-\tilde{\theta})\frac{\partial \phi}{\partial t}\dxdt -&\int_{\Omega} S(\theta_0-\tilde{\theta}_0)\phi(x,0)\dx
+ \int_Q S'(\theta-\tilde{\theta})\nabla(\theta-\tilde{\theta}) \cdot\nabla\phi \dxdt 
\\
+ &\int_Q S''(\theta-\tilde{\theta})|\nabla(\theta-\tilde{\theta})|^2\phi \dxdt
%-\int_0^T\int_{\partial\Omega}\nabla\theta\cdot\vc{n} S'(\theta- \tilde{\theta})\phi \ds\dt 
= \int_Q  \ten{G}(\theta,\ten{T}^d):\ten{T}^d S'(\theta - \tilde{\theta})\phi \dxdt
\end{split}
\end{equation}
holds for every test function $\phi\in C^{\infty}_c([-\infty,T),C^{\infty}(\Omega))$, for every function $S\in C^{\infty}(\mathbb{R})$ such that $S'\in C_0^{\infty}(\mathbb{R})$ and for $\tilde{\theta}$ which is a solution of problem
\begin{equation}
\left\{
\begin{array}{rcll}
\tilde{\theta}_t -\Delta \tilde{\theta} &=& 0 & \mbox{in } \Omega\times (0,T), \\
\frac{\partial\tilde{\theta}}{\partial\vc{n}} &=& g_{\theta} & \mbox{on } \partial\Omega\times (0,T), \\
\tilde{\theta}(x,0) &=& \tilde{\theta}_0 & \mbox{in } \Omega.
\end{array}
\right.
\label{war_brz_t_0}
\end{equation}
Furthermore, the visco-elastic strain tensor can be recovered from the equation on its evolution, i.e.
\begin{equation}
\ten{\varepsilon}^{\bf p}(x,t) = \ten{\varepsilon}^{\bf p}_0(x) + \int_0^t \ten{G}(\theta(x,\tau),\ten{T}^d(x,\tau)) \dtau,
\end{equation}
for a.e. $x\in\Omega$ and $t\in [0,T)$. Moreover $\ten{\varepsilon}^{\bf p},\ten{\varepsilon}^{\bf p}_t \in L_{M^*}(Q)$.
%\fxnote{check the spaces}
%W^{1,p'}(0,T,L^{p'}(\Omega,\mathcal{S}^3_d))$ {\color{red} Tu też trzeba zapisać to trochę inaczej!!!}.
\label{df:solution}
\end{defi}

\begin{tw}
Let initial conditions satisfy $\theta_0 \in L^1(\Omega)$, $\ten{\varepsilon}^{\bf p}_0\in L_{M^*}(\Omega,\mathcal{S}^3_d)$, boundary conditions satisfy $g_{\theta}\in L^2(0,T,L^2(\partial\Omega))$, for $p>3$ the function $\vc{g}\in L^{\infty}(0,T, W^{2,p}(\Omega,\mathbb{R}^3))$ and volume force $\vc{f}\in L^{\infty}(0,T,L^p(\Omega,\mathbb{R}^3))$, also
%\fxnote{check the spaces}
function $\ten{G}(\cdot,\cdot)$ satisfies the same conditions as in Assumption \ref{ass_G}. Then there exists a weak solution to system \eqref{full_system_2}. 

\label{th:main}
\end{tw}

The idea of proof is similar as in \cite{GKSG_NH}. We use Galerkin approximations. Usage of growth condition in Orlicz spaces instead of growth condition in Lebesgue spaces implies utilization of different analytic tools, e.g. Minty-Browder trick for Orlicz spaces, which appear here to be non-reflexive, to identify the weak limit of nonlinear term and biting limit to show the convergences in $L^1(Q)$ of right hand side in heat equation. Moreover, Young measures tools are used and some important lemmas for the Young measure are presented in Appendix \ref{C}.

\section{Proof of Theorem \ref{th:main}}

The proof of Theorem \ref{th:main} consist of few steps. Each steps is presented in separate section. 

\subsection{Transformation to a homogeneous boundary-value-problem}

First step of the proof is to transform the system into homogeneous boundary-value problem. Construction of solution is more clear in this case. Moreover, we also cut off the right hand side function in the elastic problems. Thereby, instead of volume force and boundary values we receive the same influence of exterior by using the shifts of solutions. 
%Spowoduje to usunięcie danych wejsciowych, które pojawią się tylko jako przesunięcia rozwiązań.
It allows us to focus on the important issues instead of the calculation difficulties.

Let us consider two independent systems of equations with given initial conditions and boundary data. The boundary conditions are the same as in \eqref{boun_0}. %Initial conditions $\tilde{\vc{u}}_0$ and $\tilde{\theta}_0$ belong to $L^2(\Omega,\mathbb{R}^3)$ and  $L^2(\Omega)$, respectively
\begin{equation}
\left\{
\begin{array}{rcll}
-\div \tilde{\ten{T}} &=& \vc{f} & \mbox{in } \Omega\times (0,T), \\
\tilde{\ten{T}} &=& \ten{D}\ten{\varepsilon}(\tilde{\vc{u}}) & \mbox{in } \Omega\times (0,T), \\
\tilde{\vc{u}} &=& \vc{g} & \mbox{on } \partial\Omega\times (0,T), 
%\tilde{\vc{u}}(x,0) &=& \tilde{\vc{u}}_0 & \mbox{in } \Omega. \\
\end{array}
\right.
\label{war_brz_u}
\end{equation}
and
\begin{equation}
\left\{
\begin{array}{rcll}
\tilde{\theta}_t -\Delta \tilde{\theta} &=& 0 & \mbox{in } \Omega\times (0,T), \\
\frac{\partial\tilde{\theta}}{\partial\vc{n}} &=& g_{\theta} & \mbox{on } \partial\Omega\times (0,T), \\
\tilde{\theta}(x,0) &=& \tilde{\theta}_0 & \mbox{in } \Omega.
\end{array}
\right.
\label{war_brz_t}
\end{equation}

\begin{lemat}
For $p>3$, let $\tilde{\theta}_0 \in L^2(\Omega)$, $\vc{g} \in L^{\infty}(0,T,W^{2,p}(\Omega,\mathbb{R}^3))$, $g_{\theta} \in L^2(0,T,L^2(\partial\Omega))$ and $\vc{f}\in L^{\infty}(0,T,L^p(\Omega,\mathbb{R}^3))$. Then there exists a solution to systems \eqref{war_brz_u} and \eqref{war_brz_t}. Additionally, the following estimates hold:
\begin{equation}
\begin{split}
\|\tilde{\vc{u}}\|_{L^{\infty}(0,T,W^{2,p}(\Omega))} 
& \leq 
C_1 \left(\|\vc{g}\|_{L^{\infty}(0,T,W^{2,p}(\Omega))}+ 
\|\vc{f}\|_{L^{\infty}(0,T,L^{p}(\Omega))} \right),
\\
\|\tilde{\theta}\|_{L^{\infty}(0,T,L^1(\Omega))}  + \|\tilde{\theta}\|_{L^2(0,T,W^{1,2}(\Omega))} 
& \leq 
C_2 \left(\|g_{\theta}\|_{L^2(0,T,L^2(\partial\Omega))}+\|\tilde{\theta}_0\|_{L^2(\Omega)} \right).
\nonumber
\end{split}
\end{equation}
Moreover, the following estimate holds for the Cauchy stress tensor
\begin{equation}
\|\tilde{\ten{T}}\|_{L^{\infty}(Q)} \leq 
C_3 \left(\|\vc{g}\|_{L^{\infty}(0,T,W^{2,p}(\Omega))}+ 
\|\vc{f}\|_{L^{\infty}(0,T,L^{,p}(\Omega))} \right).
\label{eq:24}
\end{equation}
\label{wyrzucenie_war_brzeg}
\end{lemat}
Results for temperature are straightforward, hence let us discuss only existence of solution to the elastic system of equations. 
\begin{proof}
Rewriting the solution in the form $\tilde{\vc{u}} = \tilde{\vc{u}}_1+\vc{g}$, instead of looking for $\tilde{\vc{u}}$ we may search for $\tilde{\vc{u}}_1$, where $\tilde{\vc{u}}_1$ is a solution of the system
\begin{equation}
\left\{
\begin{array}{rcll}
-{\div} \ten{D}\ten{\varepsilon}(\tilde{\vc{u}}_1) &=& \vc{f} +{\div} \ten{D}\ten{\varepsilon}(\vc{g}) & \mbox{in } \Omega\times (0,T), \\
%\tilde{\ten{T}} &=& \ten{D}\ten{\varepsilon}(\tilde{\vc{u}}) & \mbox{in } \Omega\times (0,T), \\
\tilde{\vc{u}}_1 &=& 0 & \mbox{on } \partial\Omega\times (0,T),
%\tilde{\vc{u}}(x,0) &=& \tilde{\vc{u}}_0 & \mbox{in } \Omega. \\
\end{array}
\right.
\label{war_brz_u_0}
\end{equation}
where function $\vc{f} +{\div} \ten{D}\ten{\varepsilon}(\vc{g})$  belongs to $L^{\infty}(0,T,L^p(\Omega,\mathbb{R}^3))$. By \cite[Theorem 7.1]{Valent} we know that there exists an unique solution to elasticity problem. By condition $p>3$ and by using the general Sobolev inequalities \cite[Theorem 6, p. 270]{Evans} we obtain the inequality \eqref{eq:24}. This estimate is crucial in the next steps of the proof.
\end{proof}

Instead of finding  $(\widehat{\vc u}, \widehat{\theta})-$ the solution to problem \eqref{full_system_2}-\eqref{init_0}-\eqref{boun_0}
we shall search for $(\vc{u}, \theta)$, where $\vc{u}=\widehat{\vc{u}}-\tilde{\vc{u}}$ and $\theta=\widehat{\theta}-\tilde{\theta}$ where $(\tilde{\vc{u}},\tilde{\theta})$ solves \eqref{war_brz_u} and \eqref{war_brz_t}. Furthermore, we consider the following system of equations

%\begin{equation}
%\left\{
%\begin{split}
%- {\div} \ten{T} = - {\div} (\widehat{\ten{T}} - \tilde{\ten{T}}) & =  0 ,
%\\
%\ten{T} & =  \ten{D}(\ten{\varepsilon}(\vc{u}) - \ten{\varepsilon}^{\bf p} ),
%\\
%\ten{\varepsilon}^{\bf p}_t & =  \ten{G}(\widehat{\theta},\widehat{\ten{T}}^d) 
%\\
%& = \ten{G}(\theta + \tilde{\theta},\ten{T}^d+\tilde{\ten{T}}^d),
%\\
%\theta_t - \Delta \theta = (\widehat{\theta} - \tilde{\theta})_t - \Delta(\widehat{\theta} - \tilde{\theta}) & = \widehat{\ten{T}}^d:\ten{G}(\theta + \tilde{\theta},\ten{T}^d+\tilde{\ten{T}}^d) 
%\\
%& = \big(\ten{T}^d + \tilde{\ten{T}}^d\big):\ten{G}(\theta + \tilde{\theta},\ten{T}^d+\tilde{\ten{T}}^d).
%\end{split}
%\right.
%\label{full_system_22}
%\end{equation}
%
%Hence, we consider the problem
\begin{equation}
\left\{
\begin{split}
- {\div} \ten{T} & =  0 ,
\\
\ten{T} & = \ten{D}(\ten{\varepsilon}(\vc{u}) - \ten{\varepsilon}^{\bf p} ),
\\
\ten{\varepsilon}^{\bf p}_t & =  \ten{G}(\theta + \tilde{\theta}, \ten{T}^d + \tilde{\ten{T}}^d),
\\
\theta_t - \Delta \theta & =   \big(\ten{T}^d + \tilde{\ten{T}}^d\big):\ten{G}(\theta + \tilde{\theta}, \ten{T}^d + \tilde{\ten{T}}^d),
\end{split}
\right.
\label{full_system_22}
\end{equation}
with boundary and initial conditions:
\begin{equation}
\left\{	
\begin{array}{rcll}
\vc{u} &=& 0 & \mbox{on } \partial\Omega\times (0,T), \\
\frac{\partial\theta}{\partial \vc{n}} &=& 0 & \mbox{on } \partial\Omega\times (0,T), \\
%\widehat{\vc{u}} &=& \vc{u}_0 - \tilde{\vc{u}}_0 = \widehat{\vc{u}}_0 & \mbox{in } \Omega \\
\theta(\cdot,0) &=& \widehat{\theta}_0 - \tilde{\theta}_0 \equiv \theta_0 & \mbox{in } \Omega, \\
\ten{\varepsilon}^{\bf p}(\cdot,0) &=& \ten{\varepsilon}^{\bf p}_0 & \mbox{in } \Omega,
\end{array}
\right.
\label{in_bou_cond}
\end{equation}
where $\widehat{\theta}_0$ is an initial condition for whole  temperature and $\tilde{\theta}_0$ is the initial condition for the system \eqref{war_brz_t}.

\subsection{Approximate solution}
Construction of approximate solutions does not differ from the one presented in \cite{GKSG_NH}. Let us define the standard truncation operator $\mathcal{T}_k(\cdot)$ by
\begin{equation}
\mathcal{T}_k(x)=\left\{
\begin{split}
k \qquad & x> k \\
x \qquad & |x|\leq k \\
-k \qquad & x <-k,
\end{split}
\right.
\label{Tk}
\end{equation}
for $k\in{\mathbb N}$. Use of truncation is implied only by integrability of the right hand side of the heat equation and initial condition for temperature. In the proof of solutions' existence we use the truncations of solution as a test function. This truncation does not need to be a linear combination of basis functions. Thus, we use two level approximation, i.e. independent approximation parameters in the displacement and temperature. Due to this construction the limit passage in each approximation level may be done independently. As the first step we pass to the limit with temperature approximations parameter, i.e. with $l\to \infty$, and latterly we pass to the limit with displacement approximation parameter. Moreover, we construct the approximate solution for visco-elastic strain tensor. After the first limit passage the visco-elastic strain tensor is an infinite dimensional approximation. The low regularity of data implies that the second limit passage requires closer attention.

Construction of approximate solution requires usage of three different bases, i.e. bases for temperature, displacement and visco-elastic strain. 

Let $\{v_i\}_{i=1}^\infty$ be the set of  eigenfunctions of  Laplace operator with the domain $W^{1,2}_n(\Omega)=\{ v\in W^{1,2}(\Omega):\quad \frac{\partial v}{\partial\vc{n}} = 0 \}$. Let $\{\mu _i \}$ be the set of corresponding eigenvalues, let  $\{v_i\}$ be orthogonal in $W^{1,2}_n(\Omega)$ and orthonormal in $L^2(\Omega)$.

To construct the basis functions for approximation let us start from considering the space $L^2(\Omega,\mathcal{S}^3)$ with the scalar product defined by
\begin{equation}
(\ten{\xi},\ten{\eta})_{\ten{D}}:=  \int_\Omega {\ten{D}}^\frac{1}{2}\ten{\xi}\cdot {\ten{D}}^\frac{1}{2}\ten{\eta} \dx %( {\ten{D}}^\frac{1}{2}\ten{\xi}, {\ten{D}}^\frac{1}{2}\ten{\eta})
\quad\mbox{for }\ten{\xi},\ten{\eta}\in L^2(\Omega,\mathcal{S}^3)
\end{equation}
where ${\ten{D}}^\frac{1}{2}\circ{\ten{D}}^\frac{1}{2}=\ten{D}$. Moreover, let $\{\vc{w}_i\}_{i=1}^{\infty}$ be the set of eigenfunctions of elasto-static operator $-{\div}\ten{D}\ten{\varepsilon}(\cdot)$ with the domain $W_0^{1,2}(\Omega,\mathbb{R}^3)$ and  $\{ \lambda_i \}$ be the corresponding eigenvalues such that $\{\vc{w}_i\}$ is orthogonal in $W^{1,2}_0(\Omega,\mathbb{R}^3)$ with the inner product
\begin{equation}
( \vc{w}, \vc{v})_{W^{1,2}_0(\Omega)}=( \ten{\varepsilon}(\vc{w}), \ten{\varepsilon}(\vc{v}))_{\ten{D}}
\end{equation}
%\begin{equation}
%\|\ten{\varepsilon}(\vc{w})\|^2_{\ten{D}}=
%%\|\vc{w}\|_{L^2(\Omega)}^2 + 
%\|\ten{D}^{\frac{1}{2}}\ten{\varepsilon}(\vc{w})\|_{L^2(\Omega)}^2,
%\end{equation}
and orthonormal in $L^2(\Omega,\mathbb{R}^3)$. Moreover, $\|\cdot\|_{\ten{D}}$ is a norm of $L^2(\Omega,\mathcal{S}^3)$, i.e.
%Hence
\begin{equation}
\|\ten{\varepsilon}(\vc{w})\|^2_{\ten{D}}=( \ten{\varepsilon}(\vc{w}), \ten{\varepsilon}(\vc{w}))_{\ten{D}}.
%\|\vc{w}\|_{L^2(\Omega)}^2 + 
%\|\ten{D}^{\frac{1}{2}}\ten{\varepsilon}(\vc{w})\|_{L^2(\Omega)}^2,
\end{equation}
Furthermore, by using the symmetry of operator $\ten{D}$ the following equality holds for basis functions $\vc{w}_i, \vc{w}_j$ 
\begin{equation}
\int_{\Omega}\ten{D}\ten{\varepsilon}(\vc{w}_i):\ten{\varepsilon}(\vc{w}_j) \dx = \lambda_i \int_{\Omega}\vc{w}_i\cdot\vc{w}_j \dx = 0,
\end{equation} 
when $i\neq j$.

%These two families of vectors shall be used to construct  the finite dimensional approximations of the displacement and the temperature.  
The idea of constructing visco-elastic strain approximations  was presented in \cite{GKSG_NH} and we hereby refer the reader to this paper for more details. We observe that $\ten{\varepsilon}(\ten{w}_i)$ are elements of $H^s(\Omega,\mathcal{S}^3)$ by regularity of eigenfunctions, where $H^s(\Omega,\mathcal{S}^3)$ is a fractional Sobolev space with scalar product denoted by $\braket{\cdot,\cdot}_s$ and $s>\frac{3}{2}$.
Let us define the orthogonal complement in $L^2(\Omega,\mathcal{S}^3)$ 
\begin{equation}\label{Vk}
V_k:= (\mbox{span}\{\ten{\varepsilon}(\ten{w}_1),...,\ten{\varepsilon}(\ten{w}_k)\})^\bot,
\end{equation}
taken with respect to the scalar product $(\cdot,\cdot)_{\ten{D}}$. Moreover, let us define 
\begin{equation}\label{Vks}
V_k^s:=V_k\cap H^s(\Omega,\mathcal{S}^3).
\end{equation}
Using \cite[Theorem 4.11]{maleknecas}, which was also used in \cite{GKSG_NH}, there exists an orthonormal basis $\{\ten{\zeta}^k_n\}_{n=1}^{\infty}$ of $V_k$ which is also an  orthogonal basis of $V_k^s$. The basis for visco-elastic strain consists of two subsets. First subset is a set of first $k$ symmetric gradients from the basis $\{\ten{w}_i\}_{i=1}^{\infty}$. The second subset consists of first $l$ functions from $\{\ten{\zeta}^k_n\}_{n=1}^{\infty}$. Thus, for each step of approximation we use $k+l$ functions to construct visco-elastic strain.

For $k,l\in\mathbb{N}$ we define
\begin{equation}
\begin{split}
\vc{u}_{k,l} & = \sum_{n=1}^k\alpha_{k,l}^n(t) \vc{w}_n,
 \\
\theta_{k,l} & = \sum_{m=1}^l\beta_{k,l}^m(t) v_m,
 \\
\ten{\varepsilon}^{\bf p}_{k,l} & = \sum_{n=1}^k\gamma_{k,l}^n(t) \ten{\varepsilon}(\vc{w}_n) + 
\sum_{m=1}^l\delta_{k,l}^m(t) \ten{\zeta}_m^k,
\end{split}
\label{eq:postac}
\end{equation}
%First index on the approximative solutions $\vc{u}_{k,l}$ and $\theta_{k,l}$ corresponds to the number of base elements used to construct approximate solutions of the displacement whereas the second one corresponds to the number of base elements which we use to construct an approximate solution of temperatures. Using two level approximation we pass to the limit with $k$ and $l$ independently. 
% Furthermore, we use Boccardo and Gallou\"{e}t approach to low regular right hand side of the heat equation.

such that $\vc{u}_{k,l}$, $\ten{\varepsilon}^{\bf p}_{k,l}$ and $\theta_{k,l}$ solve the system of equations
%such that the coefficients $\alpha_{k,l}^n(t)$, $\beta_{k,l}^n(t)$, $\gamma_{k,l}^n(t)$ and $\delta_{k,l}^n(t)$ solve the
%Then we present the approximate 
%system of equations:
%\begin{equation}
%\left\{
%\begin{array}{rcll}
%\left( \ten{T}_{k,l} , \nabla^{sym} \vc{w}_n \right) &=& \langle\vc{f},\vc{w}_n\rangle
%& n=1,..,k,
%\\
%\ten{T}_{k,l} &=& \ten{D}(\ten{\varepsilon}_{k,l} - \ten{\varepsilon}^{\bf p}_{k,l} ),
%\\
%{\color{red} \left( (\ten{\varepsilon}^{\bf p}_{k,l})_t, \nabla^{sym}\vc{w}_n\right) }&=& 
%\left( \ten{G}(\theta_{k,l},\ten{T}^d_{k,l}), \nabla^{sym}\vc{w}_n\right) 
%& n=1,..,k,
%\\
%\left((\theta_{k,l})_t,v_m\right) + \left(\nabla\theta_{k,l},\nabla v_m\right) &=& \langle \mathcal{T}_k( \ten{T}^d_{k,l}:(\ten{\varepsilon}^{\bf p}_{k,l})_t ), v_m \rangle 
%& m=1,..,l.
%\end{array}
%\right.
%\label{app_system}
%\end{equation}

\begin{equation}
\begin{array}{rll}
\int_{\Omega}  \ten{T}_{k,l} : \ten{\varepsilon}(\vc{w}_n) \dx &= 0
& n=1,...,k ,
\\[1ex]
\ten{T}_{k,l} &= \ten{D}(\ten{\varepsilon}(\vc{u}_{k,l}) - \ten{\varepsilon}^{\bf p}_{k,l} ),
\\[1ex]
\int_{\Omega}(\ten{\varepsilon}^{\bf p}_{k,l})_t : \ten{D}\ten{\varepsilon}(\vc{w}_n) \dx &= 
\int_{\Omega}\ten{G}(\tilde{\theta} + \theta_{k,l} ,  \tilde{\ten{T}}^d + \ten{T}^d_{k,l}  ) : \ten{D}\ten{\varepsilon}(\vc{w}_n) \dx 
& n=1,...,k ,
\\[1ex]
\int_{\Omega}(\ten{\varepsilon}^{\bf p}_{k,l})_t : \ten{D}\ten{\zeta}^k_m \dx &= 
\int_{\Omega}\ten{G}(\tilde{\theta} + \theta_{k,l},  \tilde{\ten{T}}^d + \ten{T}^d_{k,l}  ) : \ten{D}\ten{\zeta}^k_m \dx 
& m=1,...,l ,
\\[1ex]
\int_{\Omega}(\theta_{k,l})_t v_m\dx  + \int_{\Omega}\nabla\theta_{k,l}\cdot\nabla v_m \dx &
\\[1ex]
= \int_{\Omega} \mathcal{T}_k(  (\ten{T}_{k,l}^d + \tilde{\ten{T}}^d  ): & \ten{G}(\tilde{\theta} + \theta_{k,l} ,  \tilde{\ten{T}}^d + \ten{T}^d_{k,l} ) ) v_m \dx & m=1,...,l .
\end{array}
\label{app_system}
\end{equation}
%for each $n=1,..,k$ and $m=1,..,l$ and for the test function $\ten{\Phi}$ equal to $\ten{\varepsilon}(\vc{w}_n)$ or $\ten{\zeta}_n$.
for a.a. $t\in(0,T)$.
%{\color{red} Pewnie trzebaby dodać jeszcze, że to dla prawie każdego czasu czy dodatkowo dodać całkę po czasie?}
%and for later use we introduce the notation $\ten{T}^d=\ten{T} -\frac{1}{3}tr (\ten{T})\ten{I}$.  % and $\ten{T}^d_{k,l}=\left(\ten{D}(\ten{\varepsilon}_{k,l} - \ten{\varepsilon}^{\bf p}_{k,l}) \right) ^d $
%where $\vc{f}_k$ and $\ten{G}_k$ are the projections on the subspace spanned by the first $k$ eigenfunctions from $\{\vc{w}_n\}_{n=1}^{\infty}$ and on the subspace spanned by first $k$ eigenfunctions from $\{\nabla\vc{w}_n+\nabla^T\vc{w}_n\}_{n=1}^{\infty}$, respectively. %such that the coefficients $\alpha_{k,l}^n(t)$, $\beta_{k,l}^n(t)$, $\gamma_{k,l}^n(t)$ and $\delta_{k,l}^n(t)$ solve the
%system of equations:
Moreover, the solutions fulfill initial conditions in the following form 
\begin{equation}
\left\{
\begin{array}{rclc}
%\vc{u}_{k,l} & = & 0 & \mbox{in } \partial\Omega, \\[1ex]
%\frac{\partial \theta_{k,l}}{\partial \vc{n}} & = & 0 &  \mbox{in } \partial\Omega,\\[1ex]
%\vc{u}_{k,l}(x,0)&=&\sum_{n=1}^k \big(\vc{u}_0,\vc{w}_n\big)\vc{w}_n(x) & \forall k,l\in\mathbb{N},\\[1ex]
\left( \theta_{k,l}(x,0), v_m\right) &=& \left( \mathcal{T}_k(\theta_0),v_m \right) & m=1,..,l, \\[1ex]
\left( \ten{\varepsilon}^{\bf p}_{k,l}(x,0), \ten{\varepsilon}(\vc{w}_n) \right)_{\ten{D}} &=& \left(\ten{\varepsilon}^{\bf p}_0, \ten{\varepsilon}(\vc{w}_n) \right)_{\ten{D}}
& n=1,..,k,
\\[1ex]
\left( \ten{\varepsilon}^{\bf p}_{k,l}(x,0), \ten{\zeta}_m^k) \right)_{\ten{D}} &=& \left(\ten{\varepsilon}^{\bf p}_0, \ten{\zeta}^k_m \right)_{\ten{D}}
& m=1,..,l,
\end{array}
\right.
\label{eq:warunki_pocz_app}
\end{equation}
where $\big(\cdot,\cdot\big)$ denotes the inner product in $L^2(\Omega)$ and $\big(\cdot,\cdot\big)_{\ten{D}}$ denotes the inner product in $L^2(\Omega,\mathcal{S}^3)$.

Let us notice that $\alpha_{k,l}^n(t) =  \gamma_{k,l}^n(t)$ by the selection of Galerkin bases and representation of approximate solution \eqref{eq:postac}. To present it more clearly let 
\begin{equation}
\vc{\xi}(t) = (\beta_{k,l}^1(t),...,\beta_{k,l}^l(t),\gamma_{k,l}^1(t),..., \gamma_{k,l}^k(t),\delta_{k,l}^1(t),...,\delta_{k,l}^l(t) )^T.
\end{equation}
Then, the system of equations \eqref{app_system} may be rewritten in the form of ODE's system
\begin{equation}
\left\{
\begin{split}
(\gamma_{k,l}^n(t))_t  &= 
\frac{1}{\lambda_n} 
\int_{\Omega}\tilde{\ten{G}}(x,t,\vc{\xi}(t)) : \ten{D}\ten{\varepsilon}(\vc{w}_n) \dx  ,
\\
(\delta_{k,l}^m(t))_t  &=
\int_{\Omega}\tilde{\ten{G}}(x,t,\vc{\xi}(t)) : \ten{D}\ten{\zeta}_m^k\dx ,
\\
(\beta_{k,l}^m(t))_t &= \int_{\Omega} \mathcal{T}_k\Big(   (-(\ten{D}\sum_{m=1}^l \delta_{k,l}^m(t) \ten{\zeta}_m^k )^d + \tilde{\ten{T}}^d  \big)
 :\tilde{\ten{G}}(x,t,\vc{\xi}(t)) \Big) v_m \dx  + \mu_m \beta_{k,l}^m(t),
\end{split}
\right.
\label{app_system20}
\end{equation}
for $n=1,...,k$ and $m=1,...,l$, where
\begin{equation}
\begin{split}
\tilde{\ten{G}}(x,t,\vc{\xi}(t))
& := \ten{G}(\tilde{\theta} + \theta_{k,l},\tilde{\ten{T}}^d + \ten{T}_{k,l}^d)
\\
&=\ten{G}\Big(\sum_{m=1}^l \beta_{k,l}^m(t) v_j(x) + \tilde{\theta},  - \ten{D} \Big(\sum_{m=1}^l \delta_{k,l}^m(t) \ten{\zeta}_m^k  \Big)^d + \tilde{\ten{T}}^d  \Big).
\end{split}
\nonumber
\end{equation}
%Hence
%\begin{equation}
%\left\{
%\begin{split}
%(\gamma_{k,l}^n(t))_t  &= 
%\frac{1}{\lambda_n} 
%\int_{\Omega}\tilde{\ten{G}}(x,t,\vc{\xi}_1(t),\vc{\xi}_2(t)) : \ten{D}\ten{\varepsilon}(\vc{w}_n) \dx  ,
%\\
%(\delta_{k,l}^m(t))_t  &=
%\int_{\Omega}\tilde{\ten{G}}(x,t,\vc{\xi}_1(t),\vc{\xi}_2(t)) : \ten{D}\ten{\zeta}_m^k\dx ,
%\\
%(\beta_{k,l}^m(t))_t &= \int_{\Omega} \mathcal{T}_k\Big(  \big(( \ten{D}\sum_{n=1}^k\alpha_{k,l}^n\ten{\varepsilon}(\vc{w}_n) - \ten{D}(\sum_{n=1}^l\gamma_{k,l}^n(t) \ten{\varepsilon}(\vc{w}_n) +
%\delta_{k,l}^n(t) \ten{\zeta}_n ))^d + \tilde{\ten{T}}^d  \big)
%\\
%&\quad :\tilde{\ten{G}}(x,t,\vc{\xi}_1(t),\vc{\xi}_2(t)) \Big) v_m \dx  + \mu_m \beta_{k,l}^m(t),
%\end{split}
%\right.
%\label{app_system2}
%\end{equation}

%System \eqref{app_system} with initial conditions \eqref{eq:warunki_pocz_app} can be equivalently written  as the initial value problem 
Hence, the existence of solution to approximate system is equivalent to existence of solution to the following ODE's system 
\begin{equation}\label{47}
\begin{split}
&\frac{d\vc{\xi} }{dt}  = \vc{F}(\vc{\xi}(t),t),
\qquad
t\in [0,T),
\\
&\vc{\xi}(0) =\vc{\xi}_{0},
\end{split}
\end{equation}
where $\vc{\xi}_{0}$ is a vector of initial conditions obtained from \eqref{eq:warunki_pocz_app}. 
%Using the H\"older inequality $|\vc{F}(t)|$ can be estimate by product of $\|\ten{G}(\theta_{k,l} + \tilde{\theta},  \ten{T}^d_{k,l} + \tilde{\ten{T}}^d  )\|_{L^{p'}(\Omega)}$ and  $\|\ten{T}^d_{k,l} + \tilde{\ten{T}}^d  \|_{L^p(\Omega)}$, $\|\ten{\zeta}_n\|_{L^p(\Omega)}$ or $\|\ten{\varepsilon}(\vc{w}_n) \|_{L^{p}(\Omega)}$. In finite dimensional spaces all of the norm are equivalent, hence function $|\vc{F}(t)|$ is bounded. 
\begin{lemat}{(Existence of approximate solution)}

For initial condition satisfying $\ten{\varepsilon}^{\bf p}_0\in L_{M^*}(\Omega,\mathcal{S}^3_d)$ and $\theta_0\in L^1(\Omega)$ there exists a local solution to \eqref{47} which is absolutely continuous in time.
\label{istnienie_przyblizone}
\end{lemat}

% http://en.wikipedia.org/wiki/Carath%C3%A9odory's_existence_theorem

The proof of Lemma \ref{istnienie_przyblizone} is a consequence of application of Carath\'{e}odory Theorem, see \cite[Theorem 3.4]{maleknecas} or \cite[Appendix $(61)$]{zeidlerB}. We obtain the existence of unique absolutely continuous solution for some time interval $[0,t^*]$.

%\begin{proof}
%According to Carath\'eodory Theorem, see \cite[Theorem 3.4]{maleknecas} or \cite[Appendix $(61)$]{zeidlerB}, there exist  unique absolutely continuous functions $\beta_{k,l}^m(t)$, $\gamma_{k,l}^n(t)$ and $\delta_{k,l}^m(t)$ for every $n \leq k$ and $m \leq l$ on some time interval $[0,t^*]$. 
%%But due to continuity of $\gamma_{k,l}^n$, $\delta_{k,l}^n$ and $\beta_{k,l}^n$ and boundedness of $|\vc{F}(t)|$ we can shift $t^*$ to $T$. 
%Moreover for every $n \leq k$ there exists a unique absolutely continuous function $\alpha_{k,l}^n(t)$.% as a rescaling of the function $\gamma_{k,l}^n(t)$.
%
%
%\end{proof}

\subsection{Boundedness of energy}

Since we consider the physical model, the total energy of system should be finite. Omission of the kinetic effect implies that the total energy consists of thermal energy and potential energy. We start with consideration devoted to potential energy. The part devoted to thermal energy estimates is similar to one presented in \cite{GKSG_NH}, hence we recall the lemmas without the proofs.

\begin{defi}%{Potential energy}
We say that $\mathcal{E}(\ten{\varepsilon}(\vc{u}),\ten{\varepsilon}^{\bf p})$
is the  potential energy if
\begin{equation}
\mathcal{E}(\ten{\varepsilon}(\vc{u}),\ten{\varepsilon}^{\bf p}): = \frac{1}{2}\int_{\Omega}\ten{D}(\ten{\varepsilon}(\vc{u}) - \ten{\varepsilon}^{\bf p}):(\ten{\varepsilon}(\vc{u}) - \ten{\varepsilon}^{\bf p} ) \dx .
\nonumber
\end{equation}
\label{energia}
\end{defi}

\begin{lemat}
There exists a constant $C$ (uniform with respect to $k$ and $l$) such that 
\begin{equation}
\begin{split}
 \mathcal{E}(\ten{\varepsilon}(\vc{u}_{k,l}) &, \ten{\varepsilon}^{\bf p}_{k,l}) (t)
+ \frac{2c - d}{2}\int_{Q} M^*(x,\ten{G}(\tilde{\theta} + \theta_{k,l},\tilde{\ten{T}}^d + \ten{T}_{k,l}^d)) \dxdt
\\
&
+ c \int_Q M(x,\tilde{\ten{T}}^d + \ten{T}_{k,l}^d ) \dxdt 
\leq
C,
\end{split}
\end{equation}
where $c$ is a constant from Assumption \ref{ass_G} and $d=\min(1,c)$. Moreover, the constant $C$ depends on solution of additional problem \eqref{war_brz_u} and potential energy at the initial time
\begin{equation}
C = \int_{Q}M(x,\frac{2}{d}\tilde{\ten{T}}^d) \dxdt
+  \mathcal{E}(\ten{\varepsilon}(\vc{u}_{k,l}) , \ten{\varepsilon}^{\bf p}_{k,l})(0).
\end{equation}
\label{pom_2}
\end{lemat}

\begin{proof}
Let us start with calculating the time derivative of the potential energy $\mathcal{E}(t)$. For a.a. $t\in [0,T]$ we obtain
\begin{equation}
\begin{split}
\frac{d}{dt} \mathcal{E}(\ten{\varepsilon}(\vc{u}_{k,l}) , \ten{\varepsilon}^{\bf p}_{k,l}) 
& = 
\int_{\Omega}\ten{D}(\ten{\varepsilon}(\vc{u}_{k,l}) - \ten{\varepsilon}^{\bf p}_{k,l}):(\ten{\varepsilon}(\vc{u}_{k,l}))_t 
\dx 
\\
& \quad
-
\int_{\Omega}\ten{D}(\ten{\varepsilon}(\vc{u}_{k,l}) - \ten{\varepsilon}^{\bf p}_{k,l}): (\ten{\varepsilon}^{\bf p}_{k,l})_t \dx .
\nonumber
\end{split}
\end{equation}
The terms on the right hand side of abovementioned equation may be rewritten with usage of approximate system of equations \eqref{app_system}. Firstly, for each $n\leq k$ let us multiply \eqref{app_system}$_1$ by $(\alpha_{k,l}^n)_t$. After summing over $n=1,...,k$ we get
\begin{equation}
\int_{\Omega}  \ten{D}(\ten{\varepsilon}(\vc{u}_{k,l}) - \ten{\varepsilon}^{\bf p}_{k,l}) : (\ten{\varepsilon}(\vc{u}_{k,l}))_t \dx =0 .
\label{eq:42}
\end{equation}
Then for each $n\leq k$ let us multiply \eqref{app_system}$_3$ by $\gamma_{k,l}^n$ and for each $m\leq l$ let us multiply \eqref{app_system}$_4$ by $\delta_{k,l}^n$. Summing over $n=1,..,k$ and $m=1,...,l$ we obtain
\begin{equation}
\int_{\Omega} (\ten{\varepsilon}^{\bf p}_{k,l})_t :\ten{D}(\ten{\varepsilon}(\vc{u}_{k,l}) - \ten{\varepsilon}^{\bf p}_{k,l})\dx
=
\int_{\Omega}\ten{G}(\tilde{\theta} + \theta_{k,l} ,  \tilde{\ten{T}}^d + \ten{T}^d_{k,l}  ) : \ten{T}_{k,l} \dx .
\label{eq:43}
\end{equation}
Hence
\begin{equation}
\begin{split}
\frac{d}{dt} \mathcal{E}(\ten{\varepsilon}(\vc{u}_{k,l}) , \ten{\varepsilon}^{\bf p}_{k,l}) 
& = 
-\int_{\Omega}\ten{G}(\tilde{\theta} + \theta_{k,l},  \tilde{\ten{T}}^d + \ten{T}^d_{k,l} ) : \ten{T}_{k,l} \dx ,
\label{ene}
\end{split}
\end{equation}
and then using the property of traceless matrices we get
\begin{equation}
\begin{split}
\frac{d}{dt} \mathcal{E}(\ten{\varepsilon}(\vc{u}_{k,l}) , \ten{\varepsilon}^{\bf p}_{k,l}) 
& = 
-\int_{\Omega}\ten{G}(\tilde{\theta} + \theta_{k,l}, \tilde{\ten{T}}^d + \ten{T}^d_{k,l}  ) : (\tilde{\ten{T}}^d + \ten{T}_{k,l}^d)\dx 
\\ 
&  \quad
+\int_{\Omega}\ten{G}(\tilde{\theta} + \theta_{k,l},  \tilde{\ten{T}}^d + \ten{T}^d_{k,l} ) : \tilde{\ten{T}}^d \dx .
\nonumber
\end{split}
\end{equation}
Thus, using Assumption \ref{ass_G} and Fenchel-Young inequality we estimate the changes of potential energy by
\begin{equation}
\begin{split}
\frac{d}{dt} \mathcal{E}(\ten{\varepsilon}(\vc{u}_{k,l}) , \ten{\varepsilon}^{\bf p}_{k,l}) 
& 
\leq
- c \left( \int_{\Omega} M(x,\tilde{\ten{T}}^d + \ten{T}_{k,l}^d ) \dx
+ \int_{\Omega} M^*(x,\ten{G}(\tilde{\theta} + \theta_{k,l},\tilde{\ten{T}}^d + \ten{T}_{k,l}^d )) \dx\right)
\\
& \quad + \int_{\Omega}M(x,\frac{2}{d}\tilde{\ten{T}}^d) \dx
+ \int_{\Omega} M^* (x,\frac{d}{2}\ten{G}(\tilde{\theta} + \theta_{k,l},\tilde{\ten{T}}^d + \ten{T}_{k,l}^d)) \dx,
\nonumber
\end{split}\label{osz1}
\end{equation}
where $d=\min(1,c)$. Then by convexity of $N$-function we obtain
\begin{equation}
\begin{split}
\frac{d}{dt} \mathcal{E}(\ten{\varepsilon}(\vc{u}_{k,l}) , \ten{\varepsilon}^{\bf p}_{k,l}) 
& \leq
- c \left( \int_{\Omega} M(x,\tilde{\ten{T}}^d + \ten{T}_{k,l}^d ) \dx
+ \int_{\Omega} M^*(x,\ten{G}(\tilde{\theta} + \theta_{k,l},\tilde{\ten{T}}^d + \ten{T}_{k,l}^d)) \dx\right)
\\
& \quad + \int_{\Omega}M(x,\frac{2}{d}\tilde{\ten{T}}^d) \dx
+ \frac{d}{2}\int_{\Omega} M^* (x,\ten{G}(\tilde{\theta} + \theta_{k,l},\tilde{\ten{T}}^d + \ten{T}_{k,l}^d)) \dx .
\nonumber
\end{split}\label{osz1}
\end{equation}
Finally, after integration over time interval $(0,t)$, with $0\le t\le T$ we obtain
\begin{equation}\label{osz2}
\begin{split}
 \mathcal{E}(\ten{\varepsilon}(\vc{u}_{k,l}) , \ten{\varepsilon}^{\bf p}_{k,l}) (t)
&
+ c \int_Q M(x,\tilde{\ten{T}}^d + \ten{T}_{k,l}^d ) \dxdt
+ \frac{2c - d}{2}\int_{Q} M^*(x,\ten{G}(\tilde{\theta} + \theta_{k,l},\tilde{\ten{T}}^d + \ten{T}_{k,l}^d )) \dxdt
\\
&
\leq
\int_{Q}M(x,\frac{2}{d}\tilde{\ten{T}}^d) \dxdt
+  \mathcal{E}(\ten{\varepsilon}(\vc{u}_{k,l}) , \ten{\varepsilon}^{\bf p}_{k,l})(0),
\end{split}
\nonumber
\end{equation}
which completes the proof.
\end{proof}

\begin{uwaga} %{Uniformly bounded of the sequence $\{\ten{T}_{k,l}\}$}
From Lemma \ref{pom_2} we know that the sequence $\{\ten{T}_{k,l}^d\}$ is uniformly bounded in $L_M(Q,\mathcal{S}^3)$ with respect to $k$ and $l$, also the sequence $\{\ten{G}(\tilde{\theta} + \theta_{k,l},\tilde{\ten{T}}^d + \ten{T}_{k,l}^d )\}$ in the space $L_{M^*}(Q,\mathcal{S}^3)$  with respect to $k$ and $l$. Hence using the Fenchel-Young inequality the sequence $\{(\tilde{\ten{T}}^d + \ten{T}_{k,l}^d ):\ten{G}(\tilde{\theta} + \theta_{k,l},\tilde{\ten{T}}^d + \ten{T}_{k,l}^d )\}$ is uniformly bounded in $L^1(Q)$.
\label{wsp_ogr_T}
\end{uwaga}

\begin{uwaga}
From Lemma \ref{pom_2} we know that the sequence $\{\ten{T}_{k,l}\}$ is uniformly bounded in $L^{\infty}(0,T,L^2(\Omega,\mathcal{S}^3))$ and in particular in $L^2(0,T,L^2(\Omega,\mathcal{S}^3))$.  
\end{uwaga}

Let $P^l$ be a projection from $H^s(\Omega)$ on ${\rm lin}\{\ten{\zeta}_1^k,\ldots,\ten{\zeta}_l^k\}$, defined by  $P^l(\ten{v}):=\sum_{i=1}^{l}(\ten{v},\ten{\zeta}_i)_{\ten{D}}\ten{\zeta}_i$, then $\|P^l\varphi\|_{H^s}\le\|\varphi\|_{H^s}$. Furthermore, let $P^k$ be a projection from $H^s(\Omega)$ on ${\rm lin}\{\ten{\varepsilon}(\vc{w}_1),\ldots,\ten{\varepsilon}(\vc{w}_1)\}$, defined by 
$P^k(\ten{v}):=\sum_{i=1}^{k}(\ten{v},\ten{\varepsilon}(\vc{w}_i))_{\ten{D}}\ten{\varepsilon}(\vc{w}_i)$.
Since $P^k$ is the projection of a finite dimensional space, and the dimension of this space is independent of $l$, there exists a constant, also independent of $l$ such that 
 $\|P^k\varphi\|_{H^s}\le c\|\varphi\|_{H^s}$.
 
\begin{lemat}
The sequence $\{(\ten{\varepsilon}^{\bf p}_{k,l})_t\}$ is uniformly bounded in $L^{1}(0,T,(H^{s}(\Omega,\mathcal{S}^3))')$ with respect to $k$ and $l$. %{\em Potrzebujemy dodatkowego założenia na funkcję $G$ względem pierwszej zmiennej aby mieć jednostajną ograniczoność z względu na $k$}
\label{wsp_org_epa}
\end{lemat}

\begin{proof}
Let $\varphi\in L^{\infty}(0,T,H^{s}(\Omega,\mathcal{S}^3))$ and we may estimate as follows
\begin{equation}
\begin{split}
\int_0^T |\langle (\ten{\varepsilon}^{\bf p}_{k,l})_t, \varphi\rangle |\dt &=
\int_0^T |\langle (\ten{\varepsilon}^{\bf p}_{k,l})_t, (P^k + P^l)\varphi\rangle |\dt
\\ &
\le\int_0^T |\langle (\ten{\varepsilon}^{\bf p}_{k,l})_t, P^k\varphi\rangle |\dt
+\int_0^T |\langle (\ten{\varepsilon}^{\bf p}_{k,l})_t, P^l\varphi\rangle |\dt ,
\end{split}
\end{equation}
where the equality results from orthogonality of subspaces $\mbox{lin}\{\ten{\varepsilon}(\vc{w}_1),\ldots,\ten{\varepsilon}(\vc{w}_k)\}$ and $\mbox{lin}\{\ten{\zeta}_1^k,\ldots, \ten{\zeta}_l^k\}$. Then
\begin{equation}
\begin{split}
\int_0^T |\langle (\ten{\varepsilon}^{\bf p}_{k,l})_t, \varphi\rangle |\dt &\le
\int_0^T |\int_\Omega
\ten{G}(\tilde{\theta} + \theta_{k,l},\tilde{\ten{T}}^d + \ten{T}_{k,l}^d) P^k\varphi\dx |\dt
\\ &
\quad +
\int_0^T |\int_\Omega
\ten{G}(\tilde{\theta} + \theta_{k,l},\tilde{\ten{T}}^d + \ten{T}_{k,l}^d) P^l\varphi\dx |\dt
\\ 
& \le 
\int_0^T\|\ten{G}(\tilde{\theta} + \theta_{k,l},\tilde{\ten{T}}^d + \ten{T}_{k,l}^d)\|_{L^{1}(\Omega)}
\|P^k\varphi\|_{L^{\infty}(\Omega)}\dt\\
&
\quad + \int_0^T\|\ten{G}(\tilde{\theta} + \theta_{k,l},\tilde{\ten{T}}^d + \ten{T}_{k,l}^d )\|_{L^{1}(\Omega)}
\|P^{l}\varphi\|_{L^{\infty}(\Omega)}\dt .
\end{split}
\end{equation}
Hence $s>\frac{3}{2}$ and by Sobolev inequality $\|P^{l}\varphi\|_{L^{\infty}(\Omega)} \leq \tilde{c} \|P^{l}\varphi\|_{H^s(\Omega)}$ and $\|P^{k}\varphi\|_{L^{\infty}(\Omega)} \leq \tilde{c}\|P^{k}\varphi\|_{H^s(\Omega)}$, where $\tilde{c}$ is an optimal embedding constant. Then

\begin{equation}
\begin{split}
\int_0^T |\langle (\ten{\varepsilon}^{\bf p}_{k,l})_t, \varphi\rangle |\dt 
&\le
  \tilde c\int_0^T\|\ten{G}(\tilde{\theta} + \theta_{k,l},\tilde{\ten{T}}^d + \ten{T}_{k,l}^d )\|_{L^{1}(\Omega)}
\|P^k\varphi\|_{H^{s}(\Omega)}\dt
\\ &
\quad + \tilde c \int_0^T\|\ten{G}(\tilde{\theta} + \theta_{k,l},\tilde{\ten{T}}^d + \ten{T}_{k,l}^d)\|_{L^{1}(\Omega)}
\|P^{l}\varphi\|_{H^{s}(\Omega)}\dt
\\ &
\le c\tilde c\int_0^T\|\ten{G}(\tilde{\theta} + \theta_{k,l},\tilde{\ten{T}}^d + \ten{T}_{k,l}^d )\|_{L^{1}(\Omega)}
\|\varphi\|_{H^{s}(\Omega)}\dt
\\ &
\quad + \tilde c\int_0^T\|\ten{G}(\tilde{\theta} + \theta_{k,l},\tilde{\ten{T}}^d + \ten{T}_{k,l}^d)\|_{L^{1}(\Omega)}
\|\varphi\|_{H^{s}(\Omega)}\dt
\\
&\le 
(1+c)\tilde c\|\ten{G}(\tilde{\theta} + \theta_{k,l},\tilde{\ten{T}}^d + \ten{T}_{k,l}^d )\|_{L^{1}(Q)}
\|\varphi\|_{L^{\infty}(0,T,H^{s}(\Omega))}.
\end{split}
\end{equation}
It is obvious that $\|\ten{G}(\tilde{\theta} + \theta_{k,l},\tilde{\ten{T}}^d + \ten{T}_{k,l}^d)\|_{L^{1}(Q)}$ is bounded. Hence there exists $C>0$ such that 
\begin{equation}
\sup_{\varphi\in L^{\infty}(0,T,H^{s}(\Omega))\atop
 \|\varphi\|_{L^{\infty}(0,T,H^{s}(\Omega))}\le 1 }\int_0^T
|\langle (\ten{\varepsilon}^{\bf p}_{k,l})_t, \varphi\rangle |\dt\le C
\end{equation}
and hence the sequence $\{(\ten{\varepsilon}^{\bf p}_{k,l})_t\}$ is uniformly bounded in $L^{1}(0,T,(H^{s}(\Omega,\mathcal{S}^3))')$.
\end{proof}

%Only two kind of energy occur in the quasi-static models. 
The remaining part is to consider the internal energy of the system. Two following lemmas come from \cite{GKSG_NH}.

\begin{lemat}\label{LinftyL1}
The sequence $\{\theta_{k,l}\}$ is uniformly bounded in $L^\infty(0,T;L^1(\Omega))$ with respect to $k$ and $l$.
\end{lemat}
 
\begin{lemat}%{Energy estimate for the temperature equation}
There exists a constant $C$, depending on domain $\Omega$ and time interval $(0,T)$, such that for every $k\in\mathbb{N}$
\begin{equation}
\begin{split}
\sup_{0\leq t\leq T}&\|\theta_{k,l}(t)\|^2_{L^2(\Omega)} +
\|\theta_{k,l}\|^2_{L^2(0,T,W^{1,2}(\Omega))} +
\|(\theta_{k,l})_t\|^2_{L^2(0,T,W^{-1,2}(\Omega))}
\\
& \leq C\Big(\|\mathcal{T}_k\Big( (\tilde{\ten{T}}^d + \ten{T}_{k,l}^d  ):\ten{G}(\tilde{\theta} + \theta_{k,l},\tilde{\ten{T}}^d + \ten{T}_{k,l}^d) \Big)\|^2_{L^2(0,T,L^2(\Omega))}+
\|\mathcal{T}_k(\theta_0)\|_{L^2(\Omega)}^2\Big).
\label{numer}
\end{split}
\end{equation}
\label{lm:7}
\end{lemat}

Estimates in Lemma \ref{lm:7} depend on $k$. To complete this Section we observe that the uniform boundedness of solutions implies the global existence of approximate solutions. For each $n=1,...,k$ and $m=1,...,l$ the solutions $\{\alpha_{k,l}^n(t),\beta_{k,l}^m(t),\gamma_{k,l}^n(t),\delta_{k,l}^m(t)\}$ exist on the whole time interval $[0,T]$.

\subsection{Limit passage $l \to\infty$ and uniform estimates}
Multiplying \eqref{app_system} by time dependent test functions $\varphi_1(t),\varphi_2(t),\varphi_3(t)\in C^{\infty}([0,T])$ and $\varphi_4(t)\in C^{\infty}_c([-\infty,T])$ and then after integration over time interval $[0,T]$, we obtain the following system of equations
\begin{equation}
%\begin{array}{rll}
\begin{split}
\int_0^T\int_{\Omega}  \ten{T}_{k,l} : \ten{\varepsilon}(\vc{w}_n) \varphi_1(t)\dxdt &= 0
%& n=1,...,k ,
\\[1ex]
\ten{T}_{k,l} &= \ten{D}(\ten{\varepsilon}(\vc{u}_{k,l}) - \ten{\varepsilon}^{\bf p}_{k,l} ),
\\[1ex]
\int_0^T\int_{\Omega}(\ten{\varepsilon}^{\bf p}_{k,l})_t : \ten{\varepsilon}(\vc{w}_n) \varphi_2(t)\dxdt &= 
\int_0^T\int_{\Omega}\ten{G}(\tilde{\theta} + \theta_{k,l},  \tilde{\ten{T}}^d + \ten{T}^d_{k,l} ) : \ten{\varepsilon}(\vc{w}_m) \varphi_2(t)\dxdt 
%& m=1,...,l ,
\\[1ex]
\int_0^T\int_{\Omega}(\ten{\varepsilon}^{\bf p}_{k,l})_t : \ten{\zeta}_m^k \varphi_3(t)\dxdt &= 
\int_0^T\int_{\Omega}\ten{G}(\tilde{\theta} + \theta_{k,l},  \tilde{\ten{T}}^d + \ten{T}^d_{k,l}  ) : \ten{\zeta}_m^k \varphi_3(t)\dxdt 
%& m=1,...,l ,
\\[1ex]
-\int_0^T\int_{\Omega}\theta_{k,l} v_m\varphi_4'(t)\dxdt  
-&\int_{\Omega}\theta_{0,k,l}(x) v_m\varphi_4(0)\dx 
+ \int_0^T\int_{\Omega}\nabla \theta_{k,l}\cdot\nabla v_m \varphi_4(t)\dxdt 
\\[1ex]
= \int_0^T\int_{\Omega} \mathcal{T}_k(&(\tilde{\ten{T}}^d + \ten{T}_{k,l}^d ):  \ten{G}(\tilde{\theta} + \theta_{k,l} ,  \tilde{\ten{T}}^d + \ten{T}^d_{k,l} ) ) v_m \varphi_4(t)\dxdt 
%& m=1,...,l ,
\end{split}
%\end{array}
\label{app_system_0T}
\end{equation}
where the fist and the third equation holds for $n=1,...,k$ and the fourth and the fifth holds for $m=1,...,l$.

%{\color{blue}Using the uniform boundedness form the previous section the following convergences hold:
Firstly, we pass to the limit with $l \rightarrow \infty$ - Galerkin approximation of temperature. From the previous section we get uniform boundedness with respect to $l$ for appropriate sequences. Using appropriate subsequence, but still denoted by the indexes $k$ and $l$,  we get the following convergences
%}

\begin{equation}
\begin{array}{cl}
%\vc{u}_{k,l}\rightharpoonup \vc{u}_k & \mbox{weakly in } L^2(0,T,W^{1,2}_0(\Omega,\mathbb{R}^3)),\\
\ten{T}_{k,l}\rightharpoonup \ten{T}_k  &  \mbox{weakly in }   L^2(Q,\mathcal{S}^3),\\
\ten{T}^d_{k,l}\rightharpoonup^{*} \ten{T}^d_k  &  \mbox{weakly* in }   L_M(Q,\mathcal{S}^3_d),\\
\ten{G}(\tilde{\theta} + \theta_{k,l},\tilde{\ten{T}}^d + \ten{T}_{k,l}^d)\rightharpoonup^{*} \ten{\chi}_k  & \mbox{weakly* in } L_{M^*}(Q,\mathcal{S}^3_d) , \\
\theta_{k,l}\rightharpoonup \theta_k  &  \mbox{weakly in }   L^2(0,T,W^{1,2}(\Omega)),\\
\theta_{k,l}\rightarrow \theta_k  &  \mbox{a.e. in } \Omega \times (0,T),\\
(\ten{\varepsilon}^{\bf p}_{k,l})_t \rightharpoonup (\ten{\varepsilon}^{\bf p}_{k})_t &  \mbox{weakly in } L^{1}(0,T,(H^{s}(\Omega,\mathcal{S}^3))').
\end{array}
\label{eq:44}
\end{equation}

We pass to the limit with $l\to\infty$ in \eqref{app_system_0T} using convergences from \eqref{eq:44}.
%Using these convergences in \eqref{app_system_0T} we pass to the limit with $l\to \infty$. 
For $n=1,...,k $ and $m\in\mathbb{N}$ the following equations hold
\begin{equation}
\begin{split}
\int_0^T\int_{\Omega}  \ten{T}_{k} : \ten{\varepsilon}(\vc{w}_n) \varphi_1(t)\dxdt &= 0 ,
\\
\int_0^T\int_{\Omega}(\ten{\varepsilon}^{\bf p}_{k})_t : \ten{\varepsilon}(\vc{w}_m) \varphi_2(t)\dxdt &= 
\int_0^T\int_{\Omega}\ten{\chi}_k : \ten{\varepsilon}(\vc{w}_m) \varphi_2(t)\dxdt ,
\\
\int_0^T\int_{\Omega}(\ten{\varepsilon}^{\bf p}_{k})_t : \ten{\zeta}_m^k \varphi_3(t)\dxdt &= 
\int_0^T\int_{\Omega}\ten{\chi}_k : \ten{\zeta}_m^k \varphi_3(t)\dxdt .
\end{split}
\label{app_system_0T2a}
\end{equation}
Moreover, $\{\ten{\varepsilon}(\vc{w}_n),\ten{\zeta}_m \}_{n=1,..,k;m=1,..,\infty}$ is a base of whole space $H^s(\Omega,\mathcal{S}^3) $ then \eqref{app_system_0T2a}$_{(2)}$  and \eqref{app_system_0T2a}$_{(3)}$ equations can be replaced by 
\begin{equation}
\int_0^T\int_{\Omega}(\ten{\varepsilon}^{\bf p}_{k})_t : \ten{\zeta} \dxdt = 
\int_0^T\int_{\Omega}\ten{\chi}_k : \ten{\zeta} \dxdt 
\label{app_system_0T2_2}
\end{equation}
for $\ten{\zeta} \in L^{\infty}(0,T,H^{s}(\Omega,\mathcal{S}^3))$. To show that the \eqref{app_system_0T2_2} holds also for all $\ten{\zeta} \in L_M(Q,\mathcal{S}^3)$ we proceed similar as in \cite{G1,GSW}.

%{\color{red} $C^{\infty}(Q,\mathcal{S}^3)$ is not dense in $L_M(Q,\mathcal{S}^3)$, but for all $\ten{\zeta} \in L_M(Q,\mathcal{S}^3)$ there exist a sequence $\{\ten{\zeta}_j\} \in L^{\infty}(0,T,H^{s}(\Omega,\mathcal{S}^3))$ such that $\ten{\zeta}_j \rightharpoonup^{*} \ten{\zeta}$ weakly* in $L_M(Q,\mathcal{S}^3)$, see, e.g. \cite[131]{odAnety}. Then \eqref{app_system_0T2_2} holds for ever $\ten{\zeta}\in L_M(Q,\mathcal{S}^3)$ as a test function.}

To complete the limit passage in heat equation \eqref{app_system_0T}$_{(5)}$ we encounter the same problem as in \cite{GKSG_NH}, but here we should use different technique. It holds due to utilization the generalized Orlicz spaces instead of Lebesgue spaces. As previously, we may precisely consider the right hand side of \eqref{app_system_0T}$_{(5)}$ and this reasoning consists of three steps. The first step is to show the inequality in the Lemma \ref{lm:8}. The second step is to identify the weak limit of nonlinear term by using Minty-Browder trick. For Minty-Browder trick in nonreflexive spaces we refer the reader to \cite{Wroblewska20104136}. And finally, the last step is to show the convergence of right hand side of heat equation. 
%{\color{red}Idea of these three steps is the same as one presented in \cite{GKSG_NH}, but consideration of generalized Orlicz spaces requires that calculations are strictly different. 
%Hence, 
We present detailed calculation for these steps below.

{\bf Step 1.} {\it Limiting inequality.}

\begin{lemat}
The following inequality holds for the solution of approximate systems.
\begin{equation}
\limsup_{l\rightarrow\infty}\int_{0}^{t}\int_{\Omega}\ten{G}(\tilde{\theta} + \theta_{k,l},\tilde{\ten{T}}^d + \ten{T}^d_{k,l}):\ten{T}^d_{k,l} \dxdt \leq
\int_{0}^{t}\int_{\Omega}\ten{\chi}_k:\ten{T}^d_k \dxdt .
\label{teza-8}
\end{equation}
\label{lm:8}
\end{lemat}

\begin{proof}
Let us start with the definition of function $\psi_{\mu,\tau}:\mathbb{R}_+\to\mathbb{R}_+$. For each $\mu>0, \tau\le T-\mu, s\ge0,$ the function $\psi_{\mu,\tau}$ is defined by 
\begin{equation}\label{psi-mu}
\psi_{\mu,\tau}(s)=\left\{
\begin{array}{lcl}
1&{\rm for}&s\in[0,\tau),\\
-\frac{1}{\mu}s+\frac{1}{\mu}\tau + 1 &{\rm for}& s \in[\tau, \tau+\mu),\\
0 &{\rm for} & s\ge \tau+\mu.
\end{array}\right.
\end{equation}
We use $\psi_{\mu,\tau}(t)$ as a test function in \eqref{ene}, then after integration over time interval $(0,T)$ we get

%Initial conditions of approximative solutions converge to the initial conditions of original system. 
%We multiply  \eqref{ene} by $\psi_{\mu,t_2}(t)$ and integrate over $(0,T)$
%We multiply \eqref{ene} by $\frac{1}{\mu}$ and integrate first  over $(0,\tau)$ and  then  over $(t_2, t_2+\mu)$
%\begin{equation}\label{mu}
%\begin{split}
%\frac{1}{\mu}\int_{t_2}^{t_2+\mu}\int_0^\tau&
%\frac{d}{d t} \mathcal{E}(\ten{\varepsilon}(\vc{u}_{k,l}) , \ten{\varepsilon}^{\bf p}_{k,l}) \dt\dtau
%%\int_{\Omega}\ten{D}\big(\ten{\varepsilon}(\vc{u}_{k,l}) - \ten{\varepsilon}^{\bf p}_{k,l}\big)&:\big((\ten{\varepsilon}(\vc{u}_{k,l}) - \ten{\varepsilon}^{\bf p}_{k,l})\big)_t \dx
%%\\
%\\
%& = 
%-
%\frac{1}{\mu}\int_{t_2}^{t_2+\mu}\int_0^\tau\int_{\Omega}\ten{G}(\tilde{\theta} + \theta_{k,l},\tilde{\ten{T}}^d + \ten{T}^d_{k,l}):\ten{T}^d_{k,l}\dx \dt\dtau.
%\end{split}
%\end{equation}
%Let us first concentrate on the right hand side. 
%Next we shall use \eqref{ene}  and  multiply it  by $\psi_{\mu,t_2}(t)$ and integrate over $(0,T)$
\begin{equation}\label{mu2}
\begin{split}
\int_{0}^{T}
\frac{d}{d t} \mathcal{E}(\ten{\varepsilon}(\vc{u}_{k,l}) , \ten{\varepsilon}^{\bf p}_{k,l}) \,\psi_{\mu,\tau}\dt
\
 = 
-
\int_0^T\int_{\Omega}\ten{G}(\tilde{\theta} + \theta_{k,l},\tilde{\ten{T}}^d + \ten{T}^d_{k,l}):\ten{T}^d_{k,l} \,\psi_{\mu,\tau}\dxdt.
\end{split}
\end{equation}
Integrating by parts the left hand side of \eqref{mu2} we obtain
\begin{equation}\label{mu3}
\begin{split}
\int_{0}^{T}
\frac{d}{d t} \mathcal{E}(\ten{\varepsilon}(\vc{u}_{k,l}) , \ten{\varepsilon}^{\bf p}_{k,l}) \,\psi_{\mu,\tau}\dt
=\frac{1}{\mu}\int_{\tau}^{\tau+\mu}\mathcal{E}(\ten{\varepsilon}(\vc{u}_{k,l}) , \ten{\varepsilon}^{\bf p}_{k,l})(t) \dt-
\mathcal{E}(\ten{\varepsilon}(\vc{u}_{k,l}) , \ten{\varepsilon}^{\bf p}_{k,l})(0).
\end{split}
\end{equation}
Passing to the limit with $l\to\infty$ and using the lower semicontinuity in $L^2(0,T,L^2(\Omega;\mathcal{S}^3))$ we get
\begin{equation}\label{mu4}
\begin{split}
\liminf\limits_{l\to\infty}\int_{0}^{T}
\frac{d}{d\tau}& \mathcal{E}(\ten{\varepsilon}(\vc{u}_{k,l}) , \ten{\varepsilon}^{\bf p}_{k,l}) \,\psi_{\mu,\tau}\dt
\\
&=\liminf\limits_{l\to\infty}\frac{1}{\mu}\int_{\tau}^{\tau+\mu}\mathcal{E}(\ten{\varepsilon}(\vc{u}_{k,l}) , \ten{\varepsilon}^{\bf p}_{k,l})(t) \dt-
\lim\limits_{l\to\infty}\mathcal{E}(\ten{\varepsilon}(\vc{u}_{k,l}) , \ten{\varepsilon}^{\bf p}_{k,l})(0)
\\
&\ge \frac{1}{\mu}\int_{\tau}^{\tau+\mu}\mathcal{E}(\ten{\varepsilon}(\vc{u}_{k}) , \ten{\varepsilon}^{\bf p}_{k})(t) \dt-
\mathcal{E}(\ten{\varepsilon}(\vc{u}_{k}) , \ten{\varepsilon}^{\bf p}_{k})(0).
\end{split}\end{equation}
Comparing \eqref{mu4} and \eqref{mu2} we obtain 
\begin{equation}\label{mu5}
\begin{split}
\liminf\limits_{l\to\infty}&\left(-
\int_0^T\int_{\Omega}\ten{G}(\tilde{\theta} + \theta_{k,l},\tilde{\ten{T}}^d + \ten{T}^d_{k,l}):\ten{T}^d_{k,l} \,\psi_{\mu,\tau}\dxdt \right) 
\\
&=
\liminf\limits_{l\to\infty}\int_{0}^{T}
\frac{d}{d t} \mathcal{E}(\ten{\varepsilon}(\vc{u}_{k,l}) , \ten{\varepsilon}^{\bf p}_{k,l}) \,\psi_{\mu,\tau}\dt
\\
&\ge \frac{1}{\mu}\int_{\tau}^{\tau+\mu}\mathcal{E}(\ten{\varepsilon}(\vc{u}_{k}) , \ten{\varepsilon}^{\bf p}_{k})(t) \dt-
\mathcal{E}(\ten{\varepsilon}(\vc{u}_{k}) , \ten{\varepsilon}^{\bf p}_{k})(0),
\end{split}
\end{equation}
which is equivalent to
\begin{equation}
\begin{split}
\limsup_{l\rightarrow\infty}\int_{0}^{T} &\int_{\Omega}\ten{G}(\tilde{\theta} + \theta_{k,l},\tilde{\ten{T}}^d + \ten{T}^d_{k,l}):\ten{T}^d_{k,l} \ \psi_{\mu,\tau}\dxdt 
\\
& \leq
-\frac{1}{\mu}\int_{\tau}^{\tau+\mu}\mathcal{E}(\ten{\varepsilon}(\vc{u}_{k}) , \ten{\varepsilon}^{\bf p}_{k})(t) \dt
+\mathcal{E}(\ten{\varepsilon}(\vc{u}_{k}) , \ten{\varepsilon}^{\bf p}_{k})(0).
\label{jedna_nierownosc}
\end{split}
\end{equation}

Since $(\alpha^n_k)_t$ is not regular enough to use as a test function in \eqref{app_system}$_{(1)}$ we may use the mollifier to improve its regularity. Thus let $\eta_\epsilon$ be a standard mollifier and we mollify with respect to time. Then let us choose  $\varphi_1(t)=((\alpha^n_k)_t*\eta_{\epsilon} \mathbf{1}_{(t_1,t_2)})*\eta_{\epsilon}$ as a test function in \eqref{app_system_0T2a}$_{(1)}$ and  $\ten{\zeta}=(\ten{T}_k^d*\eta_{\epsilon} \mathbf{1}_{(t_1,t_2)})*\eta_{\epsilon}$ as a test function in \eqref{app_system_0T2_2}, then

%To complete the proof  we choose in \eqref{limit1} the test functions
% $\varphi_1(t)=((\alpha^n_k)_t*\eta_{\epsilon}\mathbf{1}_{(t_1,t_2)})*\eta_{\epsilon}$,  
%and  in \eqref{65} $\ten{\varphi}=(\ten{T}_k^d*\eta_{\epsilon}\mathbf{1}_{(t_1,t_2)})*\eta_{\epsilon}$,
%where $\eta_\epsilon$ is a standard mollifier and we mollify with respect to time. Thus we  obtain
\begin{equation}
\begin{split}
\int_Q  \ten{T}_{k} : \ten{\varepsilon}(((\alpha^n_k)_t*\eta_{\epsilon}\mathbf{1}_{(t_1,t_2)})*\eta_{\epsilon}\vc{w}_n) \dxdt &= 0,
%\\
%\int_0^T\int_{\Omega}(\ten{\varepsilon}^{\bf p}_{k})_t : \ten{D}\ten{\varepsilon}((\gamma^n_{k}*\eta_{\epsilon}\mathbf{1}_{(t_1,t_2)})*\eta_{\epsilon}\vc{w}_n) \dx &
%\\
%= 
%\int_0^T\int_{\Omega}\ten{\chi}_k : &\ten{D}\ten{\varepsilon}((\gamma^n_{k}*\eta_{\epsilon}\mathbf{1}_{(t_1,t_2)})*\eta_{\epsilon}\vc{w}_n) \dx ,
\\
\int_Q(\ten{\varepsilon}^{\bf p}_{k})_t : (\ten{T}_k^d*\eta_{\epsilon}\mathbf{1}_{(t_1,t_2)})*\eta_{\epsilon} \dxdt 
= 
\int_Q\ten{\chi}_{k} : &(\ten{T}_k^d*\eta_{\epsilon}\mathbf{1}_{(t_1, t_2)})*\eta_{\epsilon}\dxdt ,
\end{split}
\label{app_system_n}
\end{equation}
%for each $n=1,..,k$ and $m=1,..,l$ and for the test function $\ten{\Phi}$ equal to $\ten{\varepsilon}(\vc{w}_n)$ or $\ten{\zeta}_n$.
for  $n=1,...,k $. Summing \eqref{app_system_n}$_{(1)}$ over $n=1,...,k$ we obtain
\begin{equation}
\int_{t_1}^{t_2} \int_{\Omega}\ten{D}\left(\ten{\varepsilon}(\vc{u}_{k}) - \ten{\varepsilon}^{\bf p}_{k}\right)*\eta_{\epsilon}: (\ten{\varepsilon}(\vc{u}_{k})*\eta_{\epsilon})_t \dxdt =
0.
\label{pierwsze_r1}
\end{equation}
%Summing \eqref{app_system_n}$_{(2)}$ and \eqref{app_system_n}$_{(3)}$ over $n=1,...,k$ and $m\in\mathbb{N}$ we get
Moreover, using properties of traceless matrices 
\begin{equation}
\int_{t_1}^{t_2}\int_{\Omega}(\ten{\varepsilon}^{\bf p}_{k}*\eta_{\epsilon})_t:\ten{T}_{k}*\eta_{\epsilon} \dxdt =
\int_{t_1}^{t_2}\int_{\Omega}\ten{\chi}_{k}*\eta_{\epsilon}:\ten{T}_{k}*\eta_{\epsilon} \dxdt ,
\label{drugie_r2}
\end{equation}
and products in \eqref{drugie_r2} are well defined. Since for the matrices $\ten{A}\in\mathcal{S}^3_d$ and $\ten{B}\in\mathcal{S}^3$ the equivalence $\ten{A}:\ten{B}^d=\ten{A}:\ten{B}$ holds and the sequence $\{\ten{T}^d_{k}\}$ is uniformly bounded in $L_M(Q,\mathcal{S}^3_d)$. Subtracting \eqref{pierwsze_r1} and \eqref{drugie_r2} we get
\begin{equation}
\int_{t_1}^{t_2}\int_{\Omega}\ten{D}(\ten{\varepsilon}(\vc{u}_k) - \ten{\varepsilon}^{\bf p}_k)*\eta_{\epsilon}:\left((\ten{\varepsilon}(\vc{u}_k) - \ten{\varepsilon}^{\bf p}_k)*\eta_{\epsilon}\right)_t \dxdt 
 = - 
\int_{t_1}^{t_2}\int_{\Omega}\ten{\chi}_k*\eta_{\epsilon}:\ten{T}^d_k*\eta_{\epsilon} \dxdt .
\label{eq:57}
\end{equation}
Since $\ten{\varepsilon}(\vc{u}_k) - \ten{\varepsilon}^{\bf p}_k$ belongs to $L^2(Q,\mathcal{S}^3)$ we may pass to the limit with $\epsilon\to 0 $ in the left hand side of equation \eqref{eq:57}. 

To make a limit passage with $\epsilon\to 0$ on the right hand side of \eqref{eq:57} we use the lemmas presented in Appendix \ref{B}. From Lemma \ref{lm:20}, we know that 
sequences $\{M(x,\ten{T}_{k}^d*\eta_{\epsilon})\}$ and $\{M^*(x,\ten{\chi}_{k}*\eta_{\epsilon})\}$ are uniformly integrable. Moreover, $\{\ten{T}_{k}^d*\eta_{\epsilon}\}_{\epsilon}$ converge in measure to $\ten{T}_{k}^d$ and $\{\ten{\chi}_{k}*\eta_{\epsilon}\}_{\epsilon}$ convergence in measure to $\ten{\chi}_{k}$ (by Lemma \ref{lm:19}) as $\epsilon$ goes to $0$. Uniform integrability of the sequence and convergence in measure of this sequence imply (by Lemma \ref{lm:16}) that 
\begin{equation}
\begin{split}
\ten{T}_{k}^d*\eta_{\epsilon} &\xrightarrow{M} \ten{T}_{k}^d
\qquad \mbox{modularly in } L_M(Q),
\\
\ten{\chi}_{k}*\eta_{\epsilon} &\xrightarrow{M^*} \ten{\chi}_{k}
\qquad \mbox{modularly in } L_{M^*}(Q),
\end{split}
\end{equation}
as $\epsilon\to 0$.
On the basis of Lemma \ref{lm:18} we complete the limit passage on the right hand side of \eqref{eq:57}. Then we obtain the following equality 
%Tu teraz dobrze jest się zastanowić jakie mamy teraz zbieżności dla funkcji w przestrzeniach Orlicza i jak wykonac ten krok. Jest inaczej niż bylo poprzednio i trzeba to dobrze opisać.
%\begin{itemize}
%\item lemat 20: z molifajerem mamy jednostajną całkowalność
%\item lemat 19: mamy zbieżność w mierze
%\item lemat 16: zbieżność w modularze wtw gdy zbieżność w mierze i jednostajna całkowalność
%\item definicja 8 zbieżność modularna
%\item lemat \ref{lm:18}: zbieżność modularna, każdego członu implikuje zbieżnosć całek
%\end{itemize}
%i to rozumowanie powinno się zamknąć}
\begin{equation}
\frac{1}{2}\int_{\Omega}\ten{D}(\ten{\varepsilon}(\vc{u}_k) - \ten{\varepsilon}^{\bf p}_k):(\ten{\varepsilon}(\vc{u}_k) - \ten{\varepsilon}^{\bf p}_k) \dx \Big|_{t_1}^{t_2}
 = - 
\int_{t_1}^{t_2}\int_{\Omega}\ten{\chi}_k:\ten{T}^d_k \dxdt .
\label{granica_l}
\end{equation}
%{\color{red} Tu nie wiem, chyba powinno być inaczej:} 
Since $\ten\varepsilon({\vc{u}_k}),\ten{\varepsilon}^{\bf p}_k\in C_{w}([0,T],L^2(\Omega,\mathcal{S}^3))$, then we may pass
%We use Cweak ([0, T ]; L2 (Ω)) to denote functions u ∈ L∞ (0, T ; L2 (Ω)) which
%satisfy (u(t), φ) ∈ C([0, T ]) for all φ ∈ L2 (Ω). z pracy anety
with $t_1\to 0$ and conclude
\begin{equation}\label{gran}
\mathcal{E}(\ten{\varepsilon}(\vc{u}_{k}) , \ten{\varepsilon}^{\bf p}_{k})(t_2) -
\mathcal{E}(\ten{\varepsilon}(\vc{u}_{k}) , \ten{\varepsilon}^{\bf p}_{k})(0)=- 
\int_{0}^{t_2}\int_{\Omega}\ten{\chi}_k:\ten{T}^d_k \dxdt .
\end{equation}
To complete the proof let us multiply \eqref{gran} by $\frac{1}{\mu}$ and integrate over the interval $(\tau, \tau+\mu)$. \begin{equation}
\begin{split}
\frac{1}{\mu}\int_{\tau}^{\tau+\mu}\mathcal{E}(\ten{\varepsilon}(\vc{u}_{k}) , \ten{\varepsilon}^{\bf p}_{k})(s) \ds -
\mathcal{E}(\ten{\varepsilon}(\vc{u}_{k}) , \ten{\varepsilon}^{\bf p}_{k})(0)=-
\frac{1}{\mu}\int_{\tau}^{\tau+\mu}\int_{0}^{s}\int_{\Omega}\ten{\chi}_k:\ten{T}^d_k \dxdt \ds.
\label{eq:64}
\end{split}\end{equation}
For conciseness let use define the function 
\begin{equation}
F(t):=\int_{\Omega}\ten{\chi}_k:\ten{T}^d_k\dx
\end{equation} 
which belongs to $L^1(0,T)$. Then applying the Fubini theorem we obtain
%\fxnote{Sprawdzić jeszcze raz indeksy }
\begin{equation}\label{71}\begin{split}
\frac{1}{\mu}\int_{\tau}^{\tau+\mu}\int_0^s F(t) \dt\ds &=\frac{1}{\mu}\int_{\mathbb{R}^2} %\int_\mathbb{R}
\mathbf{1}_{\{0\le t\le s\}}(t)\mathbf{1}_{\{\tau\le s\le \tau+\mu\}}(s) F(t)\dt\ds
\\
&=\int_\mathbb{R} \left( \frac{1}{\mu}\int_\mathbb{R}
\mathbf{1}_{\{0\le t\le s\}}(t)\mathbf{1}_{\{\tau\le s\le \tau+\mu\}} (s)\ds\right)F(t)\dt .
\end{split}
\end{equation}
Using the definition of function $\psi_{\mu,\tau}$ we observe that
\begin{equation}\label{psi}
\psi_{\mu,\tau}(t)=\frac{1}{\mu}\int_\mathbb{R}
\mathbf{1}_{\{0\le t\le s\}}(t)\mathbf{1}_{\{\tau\le s\le \tau+\mu\}} (s)\ds.
\end{equation}
Hence, comparing \eqref{jedna_nierownosc} and \eqref{eq:64} we obtain
%\begin{equation}
%-\int_{0}^{T}\int_{\Omega}\ten{\chi}_k:\ten{T}^d_k \, \psi_{\mu,t_2}\dxdt \le
%\liminf\limits_{l\to\infty}\left( 
%-
%\int_0^T\int_{\Omega}\ten{G}(\tilde{\theta} + \theta_{k,l},\tilde{\ten{T}}^d + \ten{T}^d_{k,l}):\ten{T}^d_{k,l} \,\psi_{\mu,t_2}\dxdt \right)
%\end{equation}
%which is nothing else than %\eqref{jedna_nierownosc}. 
\begin{equation}
\limsup_{l\rightarrow\infty}\int_{0}^{T}\int_{\Omega}\ten{G}(\tilde{\theta} + \theta_{k,l},\tilde{\ten{T}}^d + \ten{T}^d_{k,l}):\ten{T}^d_{k,l} \ \psi_{\mu,\tau}(t)\dxdt \leq
\int_{0}^{T}\int_{\Omega}\ten{\chi}_k:\ten{T}^d_k \ \psi_{\mu,\tau}(t)\dxdt .
\label{jedna_nierownosc2}
\end{equation}
%{\color{red}
To complete the proof of Lemma \ref{lm:8} let us show the following estimates \begin{equation}
\begin{split}
\limsup_{l\to\infty}&\int_{0}^{t_2}\int_{\Omega}\ten{G}(\tilde{\theta} + \theta_{k,l},\tilde{\ten{T}}^d + \ten{T}^d_{k,l}):\ten{T}^d_{k,l} \dxdt
\\
&\leq
\limsup_{l\to\infty}\int_{0}^{t_2}\int_{\Omega}\ten{G}(\tilde{\theta} + \theta_{k,l},\tilde{\ten{T}}^d + \ten{T}^d_{k,l}):(\tilde{\ten{T}}^d + \ten{T}^d_{k,l}) \dxdt
\\
&
\quad 
-\lim_{l\to\infty}\int_{0}^{t_2}\int_{\Omega}\ten{G}(\tilde{\theta} + \theta_{k,l},\tilde{\ten{T}}^d + \ten{T}^d_{k,l}):\tilde{\ten{T}}^d  \dxdt
\\
& \leq
\limsup_{l\to\infty}\int_{0}^{t_2+\mu}\int_{\Omega}\ten{G}(\tilde{\theta} + \theta_{k,l},\tilde{\ten{T}}^d + \ten{T}^d_{k,l}):(\tilde{\ten{T}}^d + \ten{T}^d_{k,l})\psi_{\mu,t_2}(t) \dxdt 
\\
&\quad 
- \lim_{l\to\infty}\int_{0}^{t_2}\int_{\Omega}\ten{G}(\tilde{\theta} + \theta_{k,l},\tilde{\ten{T}}^d + \ten{T}^d_{k,l}):\tilde{\ten{T}}^d  \dxdt
\\
& \leq
\limsup_{l\rightarrow\infty}\int_{0}^{t_2+\mu}\int_{\Omega}\ten{G}(\tilde{\theta} + \theta_{k,l},\tilde{\ten{T}}^d + \ten{T}^d_{k,l}):\ten{T}^d_{k,l} \ \psi_{\mu,t_2}(t)\dxdt
\\
& \quad 
+ \lim_{l\to\infty}\int_{0}^{t_2+\mu}\int_{\Omega}\ten{G}(\tilde{\theta} + \theta_{k,l},\tilde{\ten{T}}^d + \ten{T}^d_{k,l}):\tilde{\ten{T}}^d \ \psi_{\mu,t_2}(t)\dxdt
\\
& \quad
- \lim_{l\to\infty}\int_{0}^{t_2}\int_{\Omega}\ten{G}(\tilde{\theta} + \theta_{k,l},\tilde{\ten{T}}^d + \ten{T}^d_{k,l}):\tilde{\ten{T}}^d  \dxdt
\\
& \leq
\int_{0}^{t_2+\mu}\int_{\Omega}\ten{\chi}_k:\ten{T}^d_k  \psi_{\mu,t_2}(t)\dxdt 
\\
& \quad 
+ \lim_{l\to\infty}
\int_{t_2}^{t_2+\mu}\int_{\Omega}\ten{G}(\tilde{\theta} + \theta_{k,l},\tilde{\ten{T}}^d + \ten{T}^d_{k,l}):\tilde{\ten{T}}^d  \psi_{\mu,t_2}(t)\dxdt\\
&=\int_{0}^{t_2+\mu}\int_{\Omega}\ten{\chi}_k:\ten{T}^d_k \ \psi_{\mu,t_2}(t) \dxdt
+\lim_{l\to\infty}\int_{t_2}^{t_2+\mu}\int_{\Omega}
\ten{G}(\tilde{\theta} + \theta_{k,l},\tilde{\ten{T}}^d + \ten{T}^d_{k,l}):\tilde{\ten{T}}^d_k\dxdt
\end{split}\end{equation}
Passing with $\mu\to0$ yields \eqref{teza-8}. The proof is complete.
%\fxnote{tu jest tak samo}
\end{proof}

{\bf Step 2.} {\it Minty-Browder trick.}

%Using the next Theorem we improve some of the convergences and identify the weak limits ~$\ten{\chi}_k$. 
We use the Minty-Browder trick to identify the weak limits ~$\ten{\chi}_k$. For $s\in(0,T]$ let us define $Q^s=\Omega\times(0,s)$. From the monotonicity of function $\ten{G}(\theta,\cdot)$ we obtain
\begin{equation}
\begin{split}
\int_{Q^s}\left(\ten{G}(\tilde{\theta} + \theta_{k,l},\tilde{\ten{T}}^d + \ten{T}^d_{k,l})  - \ten{G}(\tilde{\theta} + \theta_{k,l},\tilde{\ten{T}}^d + \ten{W}^d)\right) : (&\ten{T}_{k,l}^d - \ten{W}^d) \dxdt\geq 0 
\\
&
\forall\ \ten{W}^d\in L^{\infty}(Q,\mathcal{S}^3_d).
\label{eq:zal_G1}
\end{split}
\end{equation}
Repeating the procedure from the proof of Lemma \ref{lm:8} we obtain
\begin{equation}
\limsup_{n\to \infty}\int_{Q^s} \ten{G}(\tilde{\theta} + \theta_{k,l},\tilde{\ten{T}}^d + \ten{T}^d_{k,l}):\ten{T}_{k,l}^d \dxdt
 \leq
\int_{Q^s} \ten{\chi}_k:\ten{T}_{k}^d \dxdt .
\end{equation}
Moreover, using \eqref{eq:44} %{\color{red}Inny wzór powinien być }
we get %know that {\color{red} Jaką dokładnie zbieżność? G mamy zbieżne słanbo w $M^*$ a $W^d$ należy do przestrzeni $L^{\infty}(Q)$}
% bo jako, że W^d jest w L^\infty to jest też w E_M, czyli możemy mamy poniższą równość
\begin{equation}
\lim_{n\to\infty} \int_{Q^s} \ten{G}(\tilde{\theta} + \theta_{k,l},\tilde{\ten{T}}^d + \ten{T}^d_{k,l}):\ten{W}^d \dxdt
=
\int_{Q^s} \ten{\chi}_k:\ten{W}^d \dxdt .
\end{equation}
%Let us look on the rest terms from \eqref{eq:zal_G1}. Firstly, 
Pointwise convergence of $\{\theta_{k,l}\}$ implies the pointwise convergence of $\{\ten{G}(\tilde{\theta} + \theta_{k,l},\tilde{\ten{T}}^d + \ten{W}^d)\}$. Furthermore, from the Assumption \ref{ass_G} and non-negativity of $N$-functions we get
\begin{equation}
|\tilde{\ten{T}}^d + \ten{W}^d|\geq c\frac{M^*(x, \ten{G}(\tilde{\theta} + \theta_{k,l},\tilde{\ten{T}}^d + \ten{W}^d))}{|\ten{G}(\tilde{\theta} + \theta_{k,l},\tilde{\ten{T}}^d + \ten{W}^d)|}.
\end{equation}
%{\color{blue} %Tu się jeszcze nie zgadza ten kawałek 
Since $\tilde{\ten{T}}^d + \ten{W}^d$ belongs to $L^{\infty}(Q,\mathcal{S}^3_d)$ and $M^*$ is an $N$-function then the sequence $\{\ten{G}(\tilde{\theta} + \theta_{k,l},\tilde{\ten{T}}^d + \ten{W}^d)\}$ belongs to $L^{\infty}(Q,\mathcal{S}^3_d)$. By Lemma \ref{lm:16} we obtain
\begin{equation}
\ten{G}(\tilde{\theta} + \theta_{k,l},\tilde{\ten{T}}^d + \ten{W}^d) \xrightarrow{M^*} 
\ten{G}(\tilde{\theta} + \theta_{k},\tilde{\ten{T}}^d + \ten{W}^d),
\end{equation}
modularly in $L_{M^*}(Q)$.
Then %$\|\ten{G}(\tilde{\theta} + \theta_{k,l},\tilde{\ten{T}}^d + \ten{W}^d) -\ten{G}(\tilde{\theta} + \theta_{k},\tilde{\ten{T}}^d + \ten{W}^d) \|_{L,M^*}\to 0$.
\begin{equation}
\begin{split}
\int_Q | \ten{G}(\tilde{\theta} + \theta_{k,l},&\tilde{\ten{T}}^d + \ten{W}^d):(\ten{T}^d_{k,l} - \ten{W}) - \ten{G}(\tilde{\theta} + \theta_{k},\tilde{\ten{T}}^d + \ten{W}^d):(\ten{T}^d_{k} - \ten{W}^d)| \dxdt 
\\
& \leq
\int_Q | (\ten{G}(\tilde{\theta} + \theta_{k,l},\tilde{\ten{T}}^d + \ten{W}^d) - \ten{G}(\tilde{\theta} + \theta_{k},\tilde{\ten{T}}^d + \ten{W}^d)):(\ten{T}^d_{k,l} - \ten{W}^d)| \dxdt 
\\
& \quad +
\int_Q | \ten{G}(\tilde{\theta} + \theta_{k},\tilde{\ten{T}}^d + \ten{W}^d):(\ten{T}^d_{k,l} - \ten{T}^d_{k} )| \dxdt 
\end{split}
\label{eq:dod}
\end{equation}
Using H\"{o}lder inequality (Lemma \ref{lm:Cauchy}) we get
\begin{equation}
\begin{split}
\int_Q | \ten{G}(\tilde{\theta} + \theta_{k,l},&\tilde{\ten{T}}^d + \ten{W}^d):(\ten{T}^d_{k,l} - \ten{W}) - \ten{G}(\tilde{\theta} + \theta_{k},\tilde{\ten{T}}^d + \ten{W}^d):(\ten{T}^d_{k} - \ten{W}^d)| \dxdt 
\\
& \leq
2\| \ten{G}(\tilde{\theta} + \theta_{k,l},\tilde{\ten{T}}^d + \ten{W}^d) - \ten{G}(\tilde{\theta} + \theta_{k},\tilde{\ten{T}}^d + \ten{W}^d)\|_{L,M^*}
\|\ten{T}^d_{k,l} - \ten{W}^d\|_{L,M} 
\\
& \quad +
\int_Q | \ten{G}(\tilde{\theta} + \theta_{k},\tilde{\ten{T}}^d + \ten{W}^d):(\ten{T}^d_{k,l} - \ten{T}^d_{k} )| \dxdt 
\end{split}
\label{eq:dod2}
\end{equation}
Since $\|\ten{T}^d_{k,l} - \ten{W}^d\|_{L,M} $ is uniformly bounded, $\|\ten{G}(\tilde{\theta} + \theta_{k,l},\tilde{\ten{T}}^d + \ten{W}^d) -\ten{G}(\tilde{\theta} + \theta_{k},\tilde{\ten{T}}^d + \ten{W}^d) \|_{L,M^*}\to 0$ and $\ten{T}^d_{k,l} -\ten{T}^d_k \rightharpoonup 0 $ in $L_M(Q,\mathcal{S}^3_d)$ as $l$ goes to $\infty$, then the right hand side of \eqref{eq:dod2} goes to zero as $l$ goes to $\infty$. Hence
%Pointwise convergence and uniform boundedness in $L^{\infty}(Q,\mathcal{S}^3_d)$ imply
%. Using the Lebesgue dominated convergence theorem we conclude that 
%}
%%% To jest ok.  
%Czy tak? Hence, we conclude that  
%}Co my wiemy o zbiezności tego członu???
%Then, for every $\epsilon>0$ there exist $\delta>0$
%\begin{equation}
%\sup_{n\in\mathbb{N}}\int_{H} |\ten{G}(\tilde{\theta} + \theta_{k,l},\tilde{\ten{T}}^d + \ten{W}^d)|<\epsilon\qquad
%\forall\ H\subset Q \mbox{ and } |H|<\delta.
%\end{equation}  
%Then using the Vitali lemma \cite[Lemma 2.11]{maleknecas} we get the convergence
%%\begin{equation}
%%\ten{G}(\tilde{\theta} + \theta_{k,l},\tilde{\ten{T}}^d + \ten{W}^d) \to
%%\ten{G}(\tilde{\theta} + \theta_{k},\tilde{\ten{T}}^d + \ten{W}^d) \qquad \mbox{ in } L_{M^*}(Q,\mathcal{S}^3).
%%\end{equation}
%%% Mamy zbieżność w L1, zbieżność norm i granica należny do L_M
%passing to subsequence if it is necessary. Hence {\color{red} Coś nie tak?}
\begin{equation}
\begin{split}
\lim_{l\to\infty} \int_{Q^s} \ten{G}(\tilde{\theta} + \theta_{k,l},\tilde{\ten{T}}^d + \ten{W}^d) :(\ten{T}_{k,l}^d - \ten{W}^d) \dxdt
%\\
%& 
=
\int_{Q^s} \ten{G}(\tilde{\theta} + \theta_{k},\tilde{\ten{T}}^d + \ten{W}^d) :(\ten{T}_{k}^d - \ten{W}^d) \dxdt
\end{split}
\label{eq:44_2}
\end{equation}
%{\color{red} Zasadnicze pytanie czy iloczyn ciągu zbieżnego słabo w $L^\infty$ i zbieżnego mocno w $L^1$ jest zbieżny?}

Summing up, passing to the limit with $l\to \infty$ in \eqref{eq:zal_G1}, we get
\begin{equation}
\int_{Q^s} \left(\ten{\chi}_k -\ten{G}(\tilde{\theta} + \theta_{k},\tilde{\ten{T}}^d + \ten{W}^d)\right):( \ten{T}_k^d - \ten{W}^d) \dxdt \geq 0
\qquad \forall \ten{W}^d\in L^{\infty}(Q^s,\mathcal{S}^3_d).
\label{eq:test_fun0}
\end{equation}
For $i>0$ let us define the set
\begin{equation}
Q_i = \{(t,x)\in Q^s: \ |\ten{T}_{k}^d|\leq i \mbox{ a.e. in } Q^s\}.
\end{equation}
Then for $0<j<i$ and for arbitrary $h>0$ we define the function
\begin{equation}
\ten{W}^d = - \tilde{\ten{T}}^d \vc{1}_{Q^s\setminus Q_i} + \ten{T}_k^d \vc{1}_{Q_i} + h \ten{U}^d \vc{1}_{Q_j}
\label{eq:test_fun}
\end{equation}
where $\ten{U}^d\in L^{\infty}(Q,\mathcal{S}^3_d)$ and $\vc{1}_H$ is a characteristic function of set $H$. Using the function defined in \eqref{eq:test_fun} as a test function in \eqref{eq:test_fun0} we obtain
\begin{equation}
\begin{split}
\int_{Q^s} &\left(\ten{\chi}_k -\ten{G}(\tilde{\theta} + \theta_{k},\tilde{\ten{T}}^d - \tilde{\ten{T}}^d \vc{1}_{Q^s\setminus Q_i} + \ten{T}_k^d \vc{1}_{Q_i} + h \ten{U}^d \vc{1}_{Q_j}) \right) : 
\\
& ( \ten{T}_k^d + \tilde{\ten{T}}^d \vc{1}_{Q^s\setminus Q_i} - \ten{T}_k^d \vc{1}_{Q_i} - h \ten{U}^d \vc{1}_{Q_j}) \dxdt \geq 0
\end{split}
\end{equation}
%Then
%\begin{equation}
%\int_{Q^s} \left(\ten{\chi}_k -\ten{G}(\tilde{\theta} + \theta_{k},\tilde{\ten{T}}^d + \ten{T}_k^d \vc{1}_{Q_i} + h \ten{U}^d \vc{1}_{Q_j})\right):( \ten{T}_k^d - \ten{T}_k^d \vc{1}_{Q_i} - h \ten{U}^d \vc{1}_{Q_j}) \dxdt \geq 0.
%\end{equation}
Since $Q_j\subset Q_i \subset Q^s$ we get
\begin{equation}
\begin{split}
-h\int_{Q_j} &\left(\ten{\chi}_k -\ten{G}(\tilde{\theta} + \theta_{k},\tilde{\ten{T}}^d + \ten{T}_k^d + h \ten{U}^d )\right): \ten{U}^d \dxdt
\\
+\int_{Q_i \setminus Q_j} & \left(\ten{\chi}_k -\ten{G}(\tilde{\theta} + \theta_{k},\tilde{\ten{T}}^d + \ten{T}_k^d  ) \right) : 
( \ten{T}_k^d - \ten{T}_k^d ) \dxdt
\\
&+\int_{Q^s\setminus Q_i} \left(\ten{\chi}_k -\ten{G}(\tilde{\theta} + \theta_{k},\ten{0}) \right): (\tilde{\ten{T}}^d + \ten{T}^d_k ) \dxdt \geq 0.
\end{split}
\end{equation}
 %To show it we use the contradiction. Let us assume that $\ten{G}(\tilde{\theta} + \theta_{k},\ten{0})\neq 0$. 
Using %growth condition on function $\ten{G}(\theta,\cdot)$ 
Assumption \ref{ass_G} we obtain
\begin{equation}
%0 = \ten{G}(\tilde{\theta} + \theta_{k},\ten{0}):\ten{0}
%\geq 
%c\left(M(x,\ten{0})+M^*(x,\ten{G}(\tilde{\theta} + \theta_{k},\ten{0}))
%\right).
M^*(x,\ten{G}(\tilde{\theta} + \theta_{k},\ten{0}))=0
\end{equation}
and then, using Definition \ref{df:Nfunction}, we get that $\ten{G}(\tilde{\theta} + \theta_{k},\ten{0})=0$. Hence
%From the definition of $M$-function and growth conditions of function $\ten{G}(\cdot,\cdot)$ we get $\ten{G}(\tilde{\theta} + \theta_{k},0)=0$, hence 
\begin{equation}
\begin{split}
-h\int_{Q_j} &\left(\ten{\chi}_k -\ten{G}(\tilde{\theta} + \theta_{k},\tilde{\ten{T}}^d + \ten{T}_k^d + h \ten{U}^d )\right): \ten{U}^d \dxdt 
+\int_{Q^s\setminus Q_i} \ten{\chi}_k : (\tilde{\ten{T}}^d + \ten{T}^d_k ) \dxdt \geq 0.
\end{split}
\label{eq:MBT_lim}
\end{equation}
Moreover, from the definition of characteristic function
\begin{equation}
\int_{Q^s\setminus Q_i} \ten{\chi}_k : (\ten{T}^d_k + \tilde{\ten{T}}^d ) \dxdt = 
\int_{Q} \left(\ten{\chi}_k : (\tilde{\ten{T}}^d + \ten{T}^d_k)\right)\ten{1}_{Q^s\setminus Q_i} \dxdt.
\end{equation}
Since $\int_{Q} \ten{\chi}_k : (\tilde{\ten{T}}^d + \ten{T}^d_k )<\infty$ and $\left(\ten{\chi}_k : (\tilde{\ten{T}}^d + \ten{T}^d_k )\right)\ten{1}_{Q^s\setminus Q_i} \to 0$ a.e. in $Q$ as $i$ goes to $\infty$, the Lebesgue dominated convergence theorem implies that 
\begin{equation}
\lim_{i\to\infty}\int_{Q^s\setminus Q_i} \ten{\chi}_k : (\tilde{\ten{T}}^d + \ten{T}^d_k ) \dxdt  = 0.
\end{equation}
Passing to the limit with $i$ going to $\infty$ in \eqref{eq:MBT_lim} and dividing by $h$ we obtain
\begin{equation}
\int_{Q_j} \left(\ten{\chi}_k -\ten{G}(\tilde{\theta} + \theta_{k},\tilde{\ten{T}}^d + \ten{T}_k^d + h \ten{U}^d )\right): \ten{U}^d \dxdt \leq 0.
\label{qe:MBT_gr}
\end{equation}
%{\color{red} Passing to the limit with $h\to 0^+$ we use the Vitali lemma: 
Since $\tilde{\ten{T}}^d + \ten{T}_k^d + h \ten{U}^d$ goes to $\tilde{\ten{T}}^d + \ten{T}_k^d$ a.e. in $Q$ when $h\to 0^+$, $\{\ten{G}(\tilde{\theta} + \theta_{k},\tilde{\ten{T}}^d + \ten{T}_k^d + h \ten{U}^d )\}_{h>0}$ is uniformly bounded in $L_{M^*}(Q_j,\mathcal{S}^3)$, we conclude that 
%hence by the Vitali's lemma (see \cite{a}) we conclude
%\fxnote{missing cite}
\begin{equation}
\ten{G}(\tilde{\theta} + \theta_{k},\tilde{\ten{T}}^d + \ten{T}_k^d + h \ten{U}^d ) \rightharpoonup^*
\ten{G}(\tilde{\theta} + \theta_{k},\tilde{\ten{T}}^d + \ten{T}_k^d )
\end{equation}
in $L_{M^*}(Q_j,\mathcal{S}^3)$ as $h$ goes to $0^+$. Consequently, passing to the limit with $h$ going to $\infty$ in \eqref{qe:MBT_gr} we obtain
\begin{equation}
\int_{Q_j} \left(\ten{\chi}_k -\ten{G}(\tilde{\theta} + \theta_{k},\tilde{\ten{T}}^d + \ten{T}_k^d )\right): \ten{U}^d \dxdt \leq 0,
\end{equation}
for all $\ten{U}^d\in L^{\infty}(Q,\mathcal{S}^3_d)$, so taking 
\begin{equation}
\ten{U}^d =
\left\{
\begin{array}{ll}
\frac{\ten{\chi}_k -\ten{G}(\tilde{\theta} + \theta_{k},\tilde{\ten{T}}^d + \ten{T}_k^d )}{|\ten{\chi}_k -\ten{G}(\tilde{\theta} + \theta_{k},\tilde{\ten{T}}^d + \ten{T}_k^d )|} & \mbox{ when } \ten{\chi}_k \neq \ten{G}(\tilde{\theta} + \theta_{k},\tilde{\ten{T}}^d + \ten{T}_k^d ),
\\
0 & \mbox{ when } \ten{\chi}_k =\ten{G}(\tilde{\theta} + \theta_{k},\tilde{\ten{T}}^d + \ten{T}_k^d ),
\end{array}\right.
\end{equation}
we obtain
\begin{equation}
\int_{Q_j} |\ten{\chi}_k -\ten{G}(\tilde{\theta} + \theta_{k},\tilde{\ten{T}}^d + \ten{T}_k^d )| \dxdt \leq 0,
\end{equation}
i.e. $\ten{\chi}_k = \ten{G}(\tilde{\theta} + \theta_{k},\tilde{\ten{T}}^d + \ten{T}_k^d )$ a.e. in $Q^s$. From the arbitrary choices of $j>0$ and $0\leq s\leq T$ we get $\ten{\chi}_k = \ten{G}(\tilde{\theta} + \theta_{k},\tilde{\ten{T}}^d + \ten{T}_k^d )$ a.e. in $Q$.

{\bf Step 3.} {\it Limit of the right hand side of the heat equations.}

The idea for third step came from paper of Gwiazda et al.\cite{GWWKZ}. Let us start from the formulation of auxiliary lemmas which may be found with the proof in \cite{GWWKZ}. We denote by $\xrightarrow{b} $ the biting limit used in Chacon's, cf. \cite{ballmurat}. %The following lemma with a proof can be found in \cite{GWWKZ}.
%\fxnote{Wiecej cytowań}

\begin{defi}{Biting limit}

Let $\{f^{\nu}\}$ be a bounded sequence in $L^1(Q)$. We say that $f\in L^1(Q)$ is a biting limit of subsequence $\{f^{\nu}\}$ if there exist nonincreasing sequence $\{E_k\}$ with $E_k\subset Q$ and $\lim_{k\to\infty} |E_k| =0$, such that $f^{\nu}$ convergence weakly to $f$ in $L^1(Q\setminus E_k)$ for fixed $k$.
\end{defi}

\begin{lemat}
Let $a_n\in L^1(Q)$ and let $0\leq a_0\in L^1(Q)$ and
\begin{equation}
a_n \geq -a_0, 
\quad 
a_n \xrightarrow{b} a
\quad \mbox{and}
\quad
\limsup_{n\to\infty}\int_Q a_n \dxdt \leq \int_Q a\dxdt
\end{equation}
then
\begin{equation}
a_n\rightharpoonup a \quad 
\mbox{weakly in } L^1(Q).
\end{equation}
\label{lm:10}
\end{lemat}

\begin{lemat}
For each $k\in\mathbb{N}$ the sequence $\{\ten{G}(\tilde{\theta} + \theta_{k,l},\tilde{\ten{T}}^d + \ten{T}_{k,l}^d):(\tilde{\ten{T}}^d + \ten{T}_{k,l}^d )\}_{l=1}^{\infty}$ converges weakly to $\ten{G}(\tilde{\theta} + \theta_k,\tilde{\ten{T}}^d + \ten{T}_k^d):(\tilde{\ten{T}}^d + \ten{T}_k^d)$
in $L^1(Q)$.
\end{lemat}

\begin{proof}
To characterized the limit of right hand side we use the same argumentation as in \cite{GWWKZ}. Using the Assumption \ref{ass_G}, Frechet-Young inequality and convexity of $N$-functions, we get 
\begin{equation}
\begin{split}
c(  M(x,\tilde{\ten{T}}^d + &\ten{T}^d_k) + M^*(x,\ten{G}(\tilde{\theta}+\theta_{k,l},\tilde{\ten{T}}^d + \ten{T}^d_k)))
\\
& \leq
\ten{G}(\tilde{\theta}+\theta_{k,l},\tilde{\ten{T}}^d +\ten{T}_{k}^d)) : (\tilde{\ten{T}}^d +\ten{T}_{k}^d)
\\
& \leq
M(x, \frac{2}{d}(\tilde{\ten{T}}^d +\ten{T}_{k}^d)) +
M^*(x, \frac{d}{2}\ten{G}(\tilde{\theta}+\theta_{k,l},\tilde{\ten{T}}^d +\ten{T}_{k}^d)) )
\\
& \leq
M(x, \frac{2}{d}(\tilde{\ten{T}}^d +\ten{T}_{k}^d)) +
\frac{d}{2}M^*(x, \ten{G}(\tilde{\theta}+\theta_{k,l},\tilde{\ten{T}}^d +\ten{T}_{k}^d)) ),
\end{split}
\end{equation}
where $d=\min(c,1)$. And finally
\begin{equation}
\begin{split}
c M(x,\tilde{\ten{T}}^d + \ten{T}^d_k) +\frac{2c-d}{2} M^*(x,\ten{G}(\tilde{\theta}+\theta_{k,l},\tilde{\ten{T}}^d + \ten{T}^d_k))
\leq
M(x, \frac{2}{d}(\tilde{\ten{T}}^d +\ten{T}_{k}^d)).
\end{split}
\end{equation}
Hence the sequence $\{\ten{G}(\tilde{\theta}+\theta_{k,l},\tilde{\ten{T}}^d + \ten{T}^d_k)\}_{l=1}^{\infty}$ is uniformly bounded in $L_{M^*}(Q)$. % and the sequence $\{M^*(x,\ten{G}(\tilde{\theta}+\theta_{k,l},\tilde{\ten{T}}^d + \ten{T}^d_k))\}_{l=1}^{\infty}$ is uniformly bounded in $L^1(Q)$. 
Using the monotonicity of function $\ten{G}(\cdot,\cdot)$  with respect to the second variable, we get
\begin{equation}
0\leq \left(\ten{G}(\tilde{\theta}+\theta_{k,l},\tilde{\ten{T}}^d +\ten{T}_{k,l}^d) - \ten{G}(\tilde{\theta}+\theta_{k,l},\tilde{\ten{T}}^d +\ten{T}_{k}^d)\right) : (\ten{T}_{k,l}^d -\ten{T}_{k}^d).
\label{eq:92}
\end{equation}
The right hand side of abovmentioned inequality is uniformly bounded in $L^1(Q)$. 
%Moreover using the uniform boundedness of approximate solutions the sequence $\left(\ten{G}(\tilde{\theta}+\theta_{k,l},\tilde{\ten{T}}^d +\ten{T}_{k,l}) - \ten{G}(\tilde{\theta}+\theta_{k,l},\tilde{\ten{T}}^d +\ten{T}_{k})\right) : (\ten{T}_{k,l} -\ten{T}_{k})$ is uniform bounded in $L^1(Q)$.
Thus, there exists a Young measure denoted by $\mu_{x,t}(\cdot,\cdot)$, see %\cite[Corollaries 3.2-3.4]{muller},
\cite[Theorem 3.1]{muller}, such that the following convergence holds
\begin{equation}
\begin{split}
(\ten{G}(\tilde{\theta}+&\theta_{k,l},\tilde{\ten{T}}^d +\ten{T}_{k,l}^d) - \ten{G}(\tilde{\theta}+\theta_{k,l},\tilde{\ten{T}}^d +\ten{T}_{k}^d)) : (\ten{T}_{k,l}^d -\ten{T}_{k}^d)
\\
& \xrightarrow{b}
\int_{\mathbb{R}\times\mathbb{R}^{3\times 3}}
\left(\ten{G}(s,\ten{\lambda}) - \ten{G}(s,\tilde{\ten{T}}^d +\ten{T}_{k}^d)\right) : (\ten{\lambda} - (\tilde{\ten{T}}^d +\ten{T}_{k}^d))
d\mu_{x,t}(s,\ten{\lambda}).
\end{split}
\label{eq:86}
\end{equation}
as $l\to\infty$. 
%where we denote the biting limit by $\xrightarrow{b} $, for more see \cite{}\fxnote{missing cite}.% used in Chacon's, cf. \cite{}. 
%{\color{red}To co jest ważne, wiemy, że ciąg funkcji generuje miarę younga, i że mamy zbieżność tylko do takiej całki z miarą younga ale poza zbiorem miałem miary - poza epsilonowym 'odgryzionym kawałkiem'. Czyli mamy taką samą sytuację jak w przypadku biting limitu}
Using Lemma \ref{lm:charakteryzacja} we obtain 
that the measure $\mu_{x,t}(s,\ten{\lambda})$ can be presented in the form of $\delta_{\tilde{\theta}+\theta_{k}}\otimes\nu_{x,t}(\ten{\lambda})$. Then
\begin{equation}
\begin{split}
\int_{\mathbb{R}\times\mathbb{R}^{3\times 3}}
&\left(\ten{G}(s,\ten{\lambda}) - \ten{G}(s,\tilde{\ten{T}}^d +\ten{T}_{k}^d)\right) : (\ten{\lambda} - (\tilde{\ten{T}}^d +\ten{T}_{k}^d))
d\mu_{x,t}(s,\ten{\lambda})
\\
&=
\int_{\mathbb{R}^{3\times 3}}
\left(\ten{G}(\tilde{\theta}+\theta_k,\ten{\lambda}) - \ten{G}(\tilde{\theta}+\theta_k,\tilde{\ten{T}}^d +\ten{T}_{k}^d)\right) : (\ten{\lambda} - (\tilde{\ten{T}}^d +\ten{T}_{k}^d))
d\nu_{x,t}(\ten{\lambda})
\\
&=
\int_{\mathbb{R}^{3\times 3}}
\ten{G}(\tilde{\theta}+\theta_k,\ten{\lambda})  : (\ten{\lambda} - (\tilde{\ten{T}}^d +\ten{T}_{k}^d))d\nu_{x,t}(\ten{\lambda})
\\
& \quad
- \int_{\mathbb{R}^{3\times 3}}
\ten{G}
(\tilde{\theta}+\theta_k,\tilde{\ten{T}}^d +\ten{T}_{k}^d) : (\ten{\lambda} - (\tilde{\ten{T}}^d +\ten{T}_{k}^d))
d\nu_{x,t}(\ten{\lambda}).
\end{split}
\end{equation}
Since $ \int_{\mathbb{R}^{3\times 3}} \ten{\lambda}d\nu_{x,t}(\ten{\lambda}) =\tilde{\ten{T}}^d + \ten{T}_k^d$ a.e., the second term in abovementioned equation disappears. Indeed, %(Condition $(v)$ from Fundamental theorem on Young measures, see \cite[Theorem 3.1]{Muller} ).} 
%Let us consider each of these terms separately.
\begin{equation}
\begin{split}
- \int_{\mathbb{R}^{3\times 3}}
\ten{G}
(\tilde{\theta}+\theta_k,&\tilde{\ten{T}}^d +\ten{T}_{k}^d) : (\ten{\lambda} - (\tilde{\ten{T}}^d +\ten{T}_{k}^d))
d\nu_{x,t}(\ten{\lambda}) 
\\
&
= - \ten{G}
(\tilde{\theta}+\theta_k,\tilde{\ten{T}}^d +\ten{T}_{k}^d) : \left(\int_{\mathbb{R}^{3\times 3}} \ten{\lambda}d\nu_{x,t}(\ten{\lambda}) - (\tilde{\ten{T}}^d +\ten{T}_{k}^d)\right).
\end{split}
\end{equation}

%\begin{equation}
%\begin{split}
%\int_{\mathbb{R}^{3\times 3}}
%&\ten{G}(\tilde{\theta}+\theta_k,\ten{\lambda}) : (\ten{\lambda} - (\tilde{\ten{T}} +\ten{T}_{k})) 
%d\nu_{x,t}(\ten{\lambda}) 
%\\
%& =\int_{\mathbb{R}^{3\times 3}}
%\ten{G}(\tilde{\theta}+\theta_k,\ten{\lambda}) :  \ten{\lambda}
%d\nu_{x,t}(\ten{\lambda})
%-
%\int_{\mathbb{R}^{3\times 3}}
%\ten{G}(\tilde{\theta}+\theta_k,\ten{\lambda}) 
%d\nu_{x,t}(\ten{\lambda}) :  (\tilde{\ten{T}} +\ten{T}_{k})
%\\
%& =
%\ten{G}(\tilde{\theta}+\theta_k,\tilde{\ten{T}}^d+\ten{T}_k) :  (\tilde{\ten{T}}^d+\ten{T}_k)
%-
%\ten{G}(\tilde{\theta}+\theta_k,\tilde{\ten{T}}^d+\ten{T}_k) :  (\tilde{\ten{T}} +\ten{T}_{k}) 
%\\
%&= 0
%\end{split}
%\end{equation}
Moreover, the uniform boundedness of the sequence $\{\ten{G}(\tilde{\theta}+\theta_{k,l},\tilde{\ten{T}}^d +\ten{T}_{k,l}^d):(\tilde{\ten{T}}^d +\ten{T}_{k,l}^d)\}_{l=1}^{\infty}$ in $L^1(Q)$ implies that
\begin{equation}
\begin{split}
\ten{G}(\tilde{\theta}+\theta_{k,l},\tilde{\ten{T}}^d +\ten{T}_{k,l}^d):(\tilde{\ten{T}}^d +\ten{T}_{k,l}^d)
& \xrightarrow{b}
\int_{\mathbb{R}\times \mathbb{R}^{3\times 3}}
\ten{G}(s,\ten{\lambda}):\ten{\lambda} d\mu_{x,t}(s,\ten{\lambda})
\\
& =
\int_{\mathbb{R}^{3\times 3}}
\ten{G}(\tilde{\theta}+\theta_{k},\ten{\lambda}):\ten{\lambda} d\nu_{x,t}(\ten{\lambda})
\end{split}
\end{equation}
Hence, by the positivity of $\ten{G}(\tilde{\theta}+\cdot,\tilde{\ten{T}}^d +\cdot):(\tilde{\ten{T}}^d + \cdot)$ and using Lemma \ref{lm:ineq} we get
%{\color{red}powinnismy dodać dodatniość tej funkcji, żeby móc skorzystać z tych lematów} using and lemma 2.1 \cite{PG_ASG_operatory nieliniowe} Corollary 3.3 \cite{Muller} we obtain the inequality
\begin{equation}
\begin{split}
\liminf_{l\to \infty} \int_Q \ten{G}(\tilde{\theta} + \theta_{k,l},\tilde{\ten{T}}^d + \ten{T}_{k,l}^d):(\tilde{\ten{T}}^d + \ten{T}_{k,l}^d) \dxdt
 \geq \int_Q \int_{\mathbb{R}^{3\times 3}}\ten{G}(\tilde{\theta} + \theta_{k},\ten{\lambda} ):\ten{\lambda} d\nu_{x,t}(\ten{\lambda})\dxdt .
\end{split}
\end{equation}
Using Lemma \ref{lm:8} and knowing that $\ten{\chi}_k = \ten{G}(\tilde{\theta} + \theta_{k},\tilde{\ten{T}}^d + \ten{T}_k^d )$ a.e. in $Q$, we get 
\begin{equation}
\begin{split}
\int_{Q}\ten{G}(\tilde{\theta} + \theta_{k},\tilde{\ten{T}} + \ten{T}_{k}):(\tilde{\ten{T}} + \ten{T}_{k}) \dxdt
 \geq \int_Q \int_{\mathbb{R}^{3\times 3}}\ten{G}(\tilde{\theta} + \theta_{k},\ten{\lambda} ):\ten{\lambda} d\nu_{x,t}(\ten{\lambda})\dxdt .
\end{split}
\label{eq:dod3}
\end{equation}
Since $\ten{G}(\tilde{\theta} + \theta_{k},\tilde{\ten{T}}^d + \ten{T}_{k}^d) = \int_{\mathbb{R}^{3\times 3}}\ten{G}(\tilde{\theta} + \theta_{k},\ten{\lambda} ) d\nu_{x,t}(\ten{\lambda}) $ and \eqref{eq:dod3} holds we obtain that the right hand side of \eqref{eq:92} is not positive. Hence 
%\begin{equation}
%\int_{\mathbb{R}^{3\times 3}}
%\ten{G}(\tilde{\theta}+\theta_k,\tilde{\ten{T}}^d +\ten{T}_{k}): (\ten{\lambda} - (\tilde{\ten{T}} +\ten{T}_{k}))
%d\nu_{x,t}(\ten{\lambda})
%\leq 0
%\label{eq:92}
%\end{equation}
%Hence, using \eqref{eq:92} and \eqref{eq:86} (monotonicity of function $\ten{G}$) we obtain
\begin{equation}
\begin{split}
\left(\ten{G}(\tilde{\theta}+\theta_{k,l},\tilde{\ten{T}}^d +\ten{T}_{k,l}^d) - \ten{G}(\tilde{\theta}+\theta_{k,l},\tilde{\ten{T}}^d +\ten{T}_{k}^d)\right) : (\ten{T}_{k,l}^d -\ten{T}_{k}^d)
 \xrightarrow{b} 0.
\end{split}
\end{equation}
%
%{\color{red} Dlaczego tu tak jest?}
%Sequence $\{\ten{G}(\tilde{\theta}+\theta_{k,l},\tilde{\ten{T}}^d \}_{i=1}^{\infty}$ is uniformly bounded in $L_{M^*}(Q)$. Moreover this sequence convergence pointwise as a consequence of continuity of function $\ten{G}$ and pointwise convergence of the sequence $\{\tilde{\theta}+\theta_{k,l}\}_{l=1}^{\infty}$. 
%
% Po pierwsze dodanie założenia  na zależność tej funkcji od temperatury musi spowodować, że będziemy wiedzieli, że ta funkcja jest $\ten{G}(\tilde{\theta}+\theta_{k,l},\tilde{\ten{T}}^d +\ten{T}_{k})) $ jest ograniczone w jakiś sposób w $L_{M^*}$, żebyśmy mogli skorzystać z jakiegoś twierdzenia (np. Lebesgue) i mieć mocną zbieżność tego członu w $L_{M^*}$.
Using again biting limit we get
\begin{equation}
\ten{G}(\tilde{\theta}+\theta_{k,l},\tilde{\ten{T}}^d +\ten{T}_{k}^d)) : (\ten{T}_{k,l}^d -\ten{T}_{k}^d) 
 \xrightarrow{b} 0.
\end{equation}
Hence 
\begin{equation}
\begin{split}
\ten{G}(\tilde{\theta}+\theta_{k,l},\tilde{\ten{T}}^d +\ten{T}_{k,l}^d): (\ten{T}_{k}^d +\tilde{\ten{T}}^d)
 \xrightarrow{b} 
 \ten{G}(\tilde{\theta}+\theta_{k},\tilde{\ten{T}}^d +\ten{T}_{k}^d): (\ten{T}_{k}^d +\tilde{\ten{T}}^d).
\end{split}
\end{equation}
We use Lemma \ref{lm:10} to complete the proof.
%, we obtain that 
%\begin{equation}
%\ten{G}(\tilde{\theta} + \theta_{k,l},\tilde{\ten{T}}^d + \ten{T}_{k,l}^d):(\tilde{\ten{T}}^d + \ten{T}_{k,l}^d ) \rightharpoonup \ten{G}(\tilde{\theta} + \theta_k,\tilde{\ten{T}}^d + \ten{T}_k^d):(\tilde{\ten{T}}^d + \ten{T}_k^d)
%\end{equation}
%in $L^1(0,T,L^1(\Omega))$.
\end{proof}

Now, we can pass to the limit with $l\to\infty$ in \eqref{app_system_0T} with $\varphi_4(t)\in C^{\infty}([0,T]\times\Omega)$.

\begin{equation}
\begin{split}
-\int_0^T\int_{\Omega}\theta_{k} (\varphi_4(t))_t\dxdt 
&- \int_{\Omega}\theta_{k}(x,0) \varphi_4(x,0)\dx
+ \int_0^T\int_{\Omega}\nabla \theta_{k}\cdot\nabla v_m \varphi_4(t)\dxdt 
\\
= \int_0^T\int_{\Omega} \mathcal{T}_k(&(\tilde{\ten{T}}^d + \ten{T}_{k}^d ):  \ten{G}(\theta_{k} + \tilde{\theta}, \tilde{\ten{T}}^d + \ten{T}^d_{k} ) ) v_m \varphi_4(t)\dxdt 
\end{split}
\label{app_system_0T2po}
\end{equation}

We finish this section with two lemmas. We prove the uniform boundedness of the sequences $\{\ten{\varepsilon}^{\bf p}_{k}\}$ and $\{\vc{u}_k\}$ in proper spaces. This allows us to make the limit passage with second parameter in the next section.

\begin{lemat}
The sequence $\{\ten{\varepsilon}^{\bf p}_{k}\}$ is uniformly bounded in $L_{M^*}(Q,\mathcal{S}^3_d)$. Moreover, the sequence $\{(\ten{\varepsilon}^{\bf p}_{k})_t\}$ is also uniformly bounded in $L_{M^*}(Q,\mathcal{S}^3_d)$.
\label{wsp_org_ep}
\end{lemat}

\begin{proof}
%Using the ordinary equation for the visco-elastic strain tensor we show the regularity of $\ten{\varepsilon}^{\bf p}_{k,l}$. 
Let us consider the equation for the evolution of visco-elastic strain tensor
\begin{equation}
(\ten{\varepsilon}^{\bf p}_{k})_t= \ten{G}(\tilde{\theta} + \theta_{k},\tilde{\ten{T}}^d + \ten{T}_{k}^d).
\nonumber
\end{equation}
Moreover
%{\color{red} Czy nie powinno być poniżej $(...)_s$, żeby wszystko formalnie się zgadzało???}
\begin{equation}
\ten{\varepsilon}^{\bf p}_{k}(x,t) = \ten{\varepsilon}^{\bf p}_{k}(x,0) + \int_0^t(\ten{\varepsilon}^{\bf p}_{k}(x,s))_s \ds .
\nonumber
\end{equation}
Integrating the value of $M^*(x,\ten{\varepsilon}^{\bf p}_{k}(x,t))$ over cylinder $Q$ and using $\Delta_2$-condition of $N$-function $M^*$ \eqref{eq:def_delta2} we get %\eqref{eq:delta2con} (the $\Delta_2$-condition of $M$-function $M^*$)
\begin{equation}
\begin{split}
\int_Q M^*(x,\ten{\varepsilon}^{\bf p}_{k}(x,t)) \dxdt 
& \leq 	c \int_Q M^*(x,\frac{1}{2}\ten{\varepsilon}^{\bf p}_{k}(x,t)) \dxdt 
+T\int_{\Omega}h(x) \dx
\\
&
=
c\int_Q M^*(x,\frac{1}{2}\ten{\varepsilon}^{\bf p}_{k}(x,0) + \frac{1}{2}\int_0^t(\ten{\varepsilon}^{\bf p}_{k}(x,s))_s \ds))\dxdt
+T\int_{\Omega}h(x) \dx .
\end{split}
\nonumber
\end{equation}
Using the convexity of $M^*$ we obtain
\begin{equation}
\begin{split}
\int_Q M^*(x,\ten{\varepsilon}^{\bf p}_{k}(x,t)) \dxdt 
& 
\leq
\frac{c}{2}\int_Q M^*(x,\ten{\varepsilon}^{\bf p}_{k}(x,0) )\dxdt
\\
& \quad
+\frac{c}{2}\int_Q M^*(x,\int_0^t \ten{G}(\tilde{\theta}+\theta_{k},\tilde{\ten{T}}^d + \ten{T}_{k}^d) (x,s) \ds))\dxdt
+T\int_{\Omega}h(x) \dx .
\end{split}
\label{eq:lm9_1}
\end{equation}
Let us focus on the middle term on the right hand side of abovementioned equation. Changing the variable $\tau = \frac{t}{T}$ we obtain
\begin{equation}
\begin{split}
\int_0^T\int_{\Omega} & M^*(x,\int_0^t \ten{G}(\tilde{\theta}+\theta_{k},\tilde{\ten{T}}^d + \ten{T}_{k}^d) (x,s) \ds))\dxdt
\\
&=
T\int_0^1\int_{\Omega}  M^*(x,\int_0^{\tau T} \ten{G}(\tilde{\theta}+\theta_{k},\tilde{\ten{T}}^d + \ten{T}_{k}^d) (x,s) \ds))\dx \dtau .
\end{split}
\nonumber
\end{equation}
By Jensen inequality we get 
\begin{equation}
\begin{split}
T\int_0^1\int_{\Omega} & M^*(x,\int_0^t \ten{G}(\tilde{\theta}+\theta_{k},\tilde{\ten{T}}^d + \ten{T}_{k}^d) (x,s) \ds))\dxdt
\\
& \leq 
T\int_0^1 \int_{\Omega} \frac{1}{\tau T} \int_0^{\tau T}  M^*(x, \tau T \ten{G}(\tilde{\theta}+\theta_{k},\tilde{\ten{T}}^d + \ten{T}_{k}^d))) \ds \dx \dtau
\\
& \leq 
T\int_0^1 \int_{\Omega} \frac{1}{\tau T} \int_0^{\tau T}  \tau M^*(x, T \ten{G}(\tilde{\theta}+\theta_{k},\tilde{\ten{T}}^d + \ten{T}_{k}^d))) \ds \dx \dtau
\\
& =
\int_0^1 \int_{\Omega} \int_0^{\tau T}  M^*(x, T \ten{G}(\tilde{\theta}+\theta_{k},\tilde{\ten{T}}^d + \ten{T}_{k}^d))) \ds \dx \dtau .
\end{split}
\nonumber
\end{equation}
There exists $d\in\mathbb{R}$ such that $2^d\geq T$. Then using the $\Delta_2$-condition, coming back to original variable and using the Fubini theorem we get
\begin{equation}
\begin{split}
\int_0^1 \int_{\Omega} &\int_0^{\tau T}  M^*(x, T \ten{G}(\tilde{\theta}+\theta_{k},\tilde{\ten{T}}^d + \ten{T}_{k}^d))) \ds \dx \dtau
\\
& \leq
\int_0^1 \int_{\Omega} \int_0^{\tau T}  M^*(x, 2^d \ten{G}(\tilde{\theta}+\theta_{k},\tilde{\ten{T}}^d + \ten{T}_{k}^d))) \ds \dx \dtau
\\
& \leq
c^d\int_0^1 \int_{\Omega} \int_0^{\tau T}  M^*(x,  \ten{G}(\tilde{\theta}+\theta_{k},\tilde{\ten{T}}^d + \ten{T}_{k}^d))) \ds \dx \dtau
+ C(d)\int_{\Omega}h(x)\dx
\\
& =
\frac{c^d}{T}\int_0^T \int_{\Omega} \int_0^{t}  M^*(x,  \ten{G}(\tilde{\theta}+\theta_{k},\tilde{\ten{T}}^d + \ten{T}_{k}^d))) \ds \dx \dt
+ C(d)\int_{\Omega}h(x)\dx
\\
& \leq
c^d\int_0^T \int_{\Omega} M^*(x,  \ten{G}(\tilde{\theta}+\theta_{k},\tilde{\ten{T}}^d + \ten{T}_{k}^d))) \dxdt
+ C(d)\int_{\Omega}h(x)\dx .
\end{split}
\end{equation}
Coming back to \eqref{eq:lm9_1} we get
\begin{equation}
\begin{split}
\int_Q M^*(x,\ten{\varepsilon}^{\bf p}_{k}(x,t)) \dxdt 
& 
\leq
\frac{cT}{2}\int_{\Omega} M^*(x,\ten{\varepsilon}^{\bf p}_{k}(x,0) )\dx
\\
& \quad
+c^d\int_0^T \int_{\Omega} M^*(x,  \ten{G}(\tilde{\theta}+\theta_{k},\tilde{\ten{T}}^d + \ten{T}_{k}^d))) \dxdt
+ C(d)\int_{\Omega}h(x)\dx .
\end{split}
\nonumber
\end{equation}
Lemma \ref{pom_2} and initial condition in $L_{M^*}(\Omega,\mathcal{S}^3_d)$ complete the proof.
%It follows from Lemma \ref{pom_2} that the right hand side is uniformly bounded.
\end{proof}

\begin{lemat}
The sequence $\{\vc{u}_{k}\}$ is uniformly bounded in $ BD_{M^*} (\Omega,\mathbb{R}^3)$.
\label{wsp_org_u}
\end{lemat}
\begin{proof}
Let us start with showing the uniform boundedness of the sequence $\{\tenepuk\}$ in the space $L_{M^*}(Q)$. Using $\Delta_2$-condition, convexity of $N$-function and Assumption \ref{ass_G} we obtain
%From Lemma \ref{pom_2} sequence $\{\ten{T}_{k,l}\}$ is uniformly bounded in $L^2(0,T,L^2(\Omega,\mathcal{S}^3))$, then
\begin{equation}
\begin{split}
\int_Q M^*(x,\ten{\varepsilon}(\vc{u}_{k})) \dxdt
& \leq 
c \int_Q M^*(x,\frac{1}{2}\tenepuk) \dxdt + \int_Q h(x)\dxdt
\\
& =
c \int_Q M^*(x,\frac{1}{2}(\tenepuk -\tenepp_{k}) +\frac{1}{2}\tenepp_{k}) \dxdt + 
T\int_{\Omega} h(x) \dx
\\
& \leq 
\frac{c}{2}\int_Q M^*(x,\ten{\varepsilon}(\vc{u}_{k})-\tenepp_{k} ) \dxdt
+ \frac{c}{2} \int_Q M^*(x,\tenepp_{k})\dxdt
+ T\int_{\Omega} h(x) \dx
\\
& \leq
\frac{c}{2} \int_Q |\ten{\varepsilon}(\vc{u}_{k})-\tenepp_{k}|^2 \dxdt
+ \frac{c}{2} \int_Q M^*(x,\tenepp_{k}) \dxdt
+ T\int_{\Omega} h(x) \dx
\\
& \leq
\frac{c}{2} \int_Q |\ten{T}_{k}|^2 \dxdt
+ \frac{c}{2} \int_Q M^*(x,\tenepp_{k}) \dxdt
+ T\int_{\Omega} h(x) \dx .
\end{split}
\end{equation}
%Order of operation:
%\begin{itemize}
%\item $\Delta_2$-condition
%\item adding and subtracting $\tenepp$
%\item convexity of $N$-function
%\item assumption \ref{ass:Mstar_bou_by_L2}
%\item boundedness of operator $\ten{D}$ 
%\end{itemize}
%{\color{red}   
%Integrating over $Q$ and using Assumption \ref{} we get
%}
%\begin{equation}
%\begin{split}
%\int_Q M^*(x,\ten{\varepsilon}(\vc{u}_{k,l})) \dxdt
%& \leq c \int_0^T\int_{\Omega} |\ten{T}_{k,l}|^{p'} \dxdt
%+ c\int_0^T\int_{\Omega} |\ten{\varepsilon}^{\bf p}_{k,l}|^{p'} \dxdt
%\\
%& \leq c \int_0^T\int_{\Omega} |\ten{T}_{k,l}|^{2} \dxdt
%+ c\int_0^T\int_{\Omega} |\ten{\varepsilon}^{\bf p}_{k,l}|^{p'} \dxdt
%\\
%& \leq c \|\ten{T}_{k,l}\|^2_{L^2(0,T,L^2(\Omega))}
%+ c\|\ten{\varepsilon}^{\bf p}_{k,l}\|^{p'}_{L^{p'}(0,T,L^{p'}(\Omega))}  .
%\end{split}
%\end{equation}
%which completes the proof. 
%It follows from Lemma \ref{wsp_ogr_T} and Lemma \ref{wsp_org_ep} that the sequence $\{\ten{\varepsilon}_{k,l}\}$ is uniformly bounded in $L^2(0,T,L^2(\Omega,\mathcal{S}^3))$. 
Following Anzellotti and Giaquinta \cite[Preposition 1.2 a)]{AnzellottiGiaquinta} we get the inequality 
\begin{equation}
\|\vc{u}_k\|_{L^1(Q)} \leq C \|\ten{\varepsilon}(\vc{u}_k)\|_{L^1(Q)},
\nonumber
\end{equation}
where $C$ is a constant depending on $\Omega$. Finally we get the estimates
\begin{equation}
\|\vc{u}_k\|_{L^1(Q)} \leq C\int_Q M^*(x,\ten{\varepsilon}(\vc{u}_k))\dxdt ,
\nonumber
\end{equation}
which completes the proof.
%%\fxnote{Tu jeszcze nie koniec}
%%{\color{red} To complete the proof, let us use the Korn inequality 
%%(cf.~\cite[Theorem 1.10]{maleknecas}) we conclude that $\{\vc{u}_{k}\}$ is uniformly bounded in $BD_{M^*}(Q,\mathbb{R}^3)$.
%%
%%Aby dokończyć dowód musimy pokoazać, że $\vc{u}$ należy do $L^1(Q)$.}
 
%{\color{red} Bo mamy też jednostajnie ograniczone w $L^1$ i tu nierówność Korna daje nam ograniczoność w Bounded deformation}

%note that $\ten{\varepsilon}(\vc{u}_{k,l})$ is the symmetric gradient of the displacement, so using the  Korn inequality 
%(cf.~\cite[Theorem 1.10]{maleknecas}) we conclude that the sequence $\{\vc{u}_{k,l}\}$ is uniformly bounded in $L^{p'}(0,T,W^{1,p'}_0(\Omega,\mathbb{R}^3))$. 
% Poincar\'{e} inequality Evans page 271 Theorem 3
% Korn inequality w książce Josefa strona 192 ale brzeg klasy C^{0,1} (strona 17 Valent) , czyli jest lipshitzowski
\end{proof}

\subsection{Limit passage $k \to\infty$ }\label{lim_k}

The considerations over the second limit passage we start from discussing the existence of heat equation solution. In the Appendix \ref{A}, we prove the existence of renormalised solution to parabolic equation with Neumann boundary condition, which is an extension of results presented by Blanchard and Murat in series of papers. In \cite{Blanchard, BlanchardMurat} the existence and uniqueness of renormalised solution is proved in the case of Dirichlet boundary condition. 

Uniform boundedness presented in previous sections gives us the following convergences  
\begin{equation}
\begin{array}{cl}
\vc{u}_{k}\rightharpoonup \vc{u} & \mbox{weakly in } BD_{M^*}(Q,\mathbb{R}^3),\\
\ten{T}_{k}\rightharpoonup \ten{T}  &  \mbox{weakly in }   L^2(Q,\mathcal{S}^3),\\
\ten{T}^d_{k}\rightharpoonup^{*} \ten{T}^d  &  \mbox{weakly* in }   L_M(Q,\mathcal{S}^3_d),\\
\ten{G}(\tilde{\theta} + \theta_{k},\tilde{\ten{T}}^d + \ten{T}_{k}^d)\rightharpoonup^{*} \ten{\chi}  & \mbox{weakly* in } L_{M^*}(Q,\mathcal{S}^3_d), \\
%\mathcal{T}_K(\theta_{k})\rightarrow \mathcal{T}_K(\theta)  &  \mbox{ in }   L^2(0,T,W^{1,2}(\Omega)),\\
%\theta_{k}\rightarrow \theta  &  \mbox{a.e. in } \Omega \times (0,T),\\
(\ten{\varepsilon}^{\bf p}_{k})_t \rightharpoonup^{*} (\ten{\varepsilon}^{\bf p})_t &  \mbox{weakly* in } L_{M^*}(Q,\mathcal{S}^3_d).
\end{array}
\label{eq:104a}
\end{equation}
Then as described in Appendix \ref{A} we also have convergences
\begin{equation}
\begin{array}{cl}
\mathcal{T}_K(\theta_{k})\rightarrow \mathcal{T}_K(\theta)  &  \mbox{ in }   L^2(0,T,W^{1,2}(\Omega)),\\
\theta_{k}\rightarrow \theta  &  \mbox{a.e. in } \Omega \times (0,T),
\end{array}
\label{eq:104b}
\end{equation}
for every $K>0$. Using these convergences in  %\eqref{app_system_0T2_2},  
\eqref{app_system_0T2a}$_{(1)}$ and \eqref{app_system_0T2_2}, we get
\begin{equation}
\begin{split}
\int_Q  \ten{T} :  \nabla\ten{\varphi}\dxdt &= 0
\\
\int_Q(\ten{\varepsilon}^{\bf p})_t : \ten{\psi}\dxdt &= 
\int_Q\ten{\chi} : \ten{\psi}\dxdt 
\end{split}
\label{app_system_0T2a2}
\end{equation}
for $\ten{\varphi}\in C^{\infty}([0,T],L^2(\Omega,\mathbb{R}^3))$ and $\ten{\psi}\in L_M(Q,\mathcal{S}^3)$. %{\color{red} Jeszcze raz gestość funkcji}

To complete the limit passage we deal with the same problem as in the previous step, i.e. we have to identify the limit of the right hand side of heat equation. Once again, the identification of this limit contains of three steps.

In the proof of the following lemma we proceed similarly as in the proof of Lemma \ref{lm:8}.

\begin{lemat}
The following inequality holds for the solution of approximate systems.
\begin{equation}
\limsup_{k\rightarrow\infty}\int_{0}^{t_2}\int_{\Omega}\ten{G}(\tilde{\theta} + \theta_{k},\tilde{\ten{T}}^d + \ten{T}^d_{k}):\ten{T}^d_k \dxdt \leq
\int_{0}^{t_2}\int_{\Omega}\ten{\chi}:\ten{T}^d \dxdt .
\label{jedna_nierownosc_1}
\end{equation}
\end{lemat}

\begin{proof}
Using the lower semicontinuity in $L^2(Q)$ we get
%Due to \eqref{82} we can rewrite \eqref{gran} as follows
%\begin{equation}\label{do-l}
%\frac{d}{dt} \mathcal{E}(\ten{\varepsilon}(\vc{u}_{k}) , \ten{\varepsilon}^{\bf p}_{k}) 
% = 
%-
%\int_{\Omega}\ten{G}(\tilde{\theta} + \theta_{k},\tilde{\ten{T}}^d + \ten{T}^d_{k}):\ten{T}^d_{k}\dx.
%\end{equation}
%We multiply the above identity by $\psi_{\mu,t_2}$ given by formula \eqref{psi-mu} and integrate over $(0,T)$. 
%Passing to  the limit $k\to\infty$ we proceed in the same manner as in the proof of Lemma~\ref{lm:8} and obtain
\begin{equation}\label{mu4a}
\begin{split}
\liminf\limits_{k\to\infty}\int_{0}^{T}
\frac{d}{d t}& \mathcal{E}(\ten{\varepsilon}(\vc{u}_{k}) , \ten{\varepsilon}^{\bf p}_{k}) \,\psi_{\mu,\tau}\dt
\\
& =\liminf\limits_{k\to\infty}\frac{1}{\mu}\int_{\tau}^{\tau+\mu}\mathcal{E}(\ten{\varepsilon}(\vc{u}_{k}) , \ten{\varepsilon}^{\bf p}_{k})(t) \dt-
\lim\limits_{k\to\infty}\mathcal{E}(\ten{\varepsilon}(\vc{u}_{k}) , \ten{\varepsilon}^{\bf p}_{k})(0)\\
&\ge \frac{1}{\mu}\int_{\tau}^{\tau+\mu}\mathcal{E}(\ten{\varepsilon}(\vc{u}_{k}) , \ten{\varepsilon}^{\bf p}_{k})(t) \dt-
\mathcal{E}(\ten{\varepsilon}(\vc{u}) , \ten{\varepsilon}^{\bf p})(0).
\end{split}\end{equation}

%%For the final step of the proof of the lemma we need to show that  the energy equality holds. Contrary to the case of previous section, we cannot use the time derivative of the limit, namely $\vc{\varepsilon}(\vc{u})_t$ as the test function. Although we shall mollifty with respect to time, but  the regularity with respect to space is not sufficient since possibly $p'< 2$. Therefore we proceed differently. We use an approximate sequence as a test function in the limit identity. Indeed, we take in \eqref{limit1a} %and\eqref{65a} 
%%the test function $\vc{\varphi}=(\ten{\varepsilon}(\vc{u}_k)*\eta_\epsilon)_t\mathbf{1}_{(t_1,t_2)})*\eta_{\epsilon}$, 
%%where again $\eta_\epsilon$ is a standard mollifier and we mollify with respect to time

%{\color{red} Czy ta regularność jest dobra aby użyć tej funkcji jak funkcji testującej?} tak, bo mamy u_k w odpowiedniej przestrzeni
We use $\ten{\varphi}_1= ((\ten{\varepsilon}(\vc{u}_k)*\eta_{\epsilon})_t \mathbf{1}_{(t_1,t_2)})*\eta_{\epsilon}$, where $\eta_{\epsilon}$ is a standard mollifier with respect to time,5 as a test function in \eqref{app_system_0T2a2} then
\begin{equation}
\int_{t_1}^{t_2} \int_{\Omega}\ten{D}(\ten{\varepsilon}(\vc{u}) - \ten{\varepsilon}^{\bf p})*\eta_{\epsilon}: (\ten{\varepsilon}(\vc{u}_k)*\eta_{\epsilon})_t \dxdt =
0.
\label{pierwsze_r2}
\end{equation}

Moreover, we use $\ten{\psi}=(\ten{T}^d*\eta_{\epsilon} \mathbf{1}_{(t_1,t_2)})*\eta_{\epsilon}$ as a test function in \eqref{app_system_0T2_2}. Then
% In a consequence we obtain \eqref{drugie_r2}, which together with \eqref{82} yields
\begin{equation}
\int_{t_1}^{t_2}\int_{\Omega}(\ten{\varepsilon}^{\bf p}_{k}*\eta_{\epsilon})_t:\ten{T}*\eta_{\epsilon} \dxdt =
\int_{t_1}^{t_2}\int_{\Omega}\ten{G}(\tilde{\theta}+\theta_k,\tilde{\ten{T}}^d+\ten{T}^d_k)*\eta_{\epsilon}:\ten{T}*\eta_{\epsilon} \dxdt .
\label{drugie_r2a}
\end{equation}
Products in \eqref{drugie_r2a} are well defined. %, since for the matrices $\ten{A}\in\mathcal{S}^3_d$ and $\ten{B}\in\mathcal{S}^3$ the equivalence $\ten{A}:\ten{B}^d=\ten{A}:\ten{B}$ holds and tensor $\ten{T}^d$ belongs to $E_M(Q)$.%^{p'}(0,T,L^{p'}(\Omega,\mathcal{S}^3))$.
Subtracting  these two equations we get
\begin{equation}
\int_{t_1}^{t_2}\int_{\Omega}\ten{T}*\eta_{\epsilon}:(\ten{\varepsilon}(\vc{u}_k) - \ten{\varepsilon}^{\bf p}_k)_t*\eta_{\epsilon} \dxdt=
- 
\int_{t_1}^{t_2}\int_{\Omega}\ten{G}(\tilde{\theta}+\theta_k,\tilde{\ten{T}}^d + \ten{T}^d_k)*\eta_{\epsilon}:\ten{T}^d*\eta_{\epsilon} \dxdt .
\label{granica_k_ptrzed}
\end{equation}
For every $\epsilon>0$ the sequence $\{(\ten{\varepsilon}(\vc{u}_k) - \ten{\varepsilon}^{\bf p}_k)_t*\eta_{\epsilon}\}$ belongs to $L^2(Q,\mathcal{S}^3)$ and is uniformly bounded in $L^2(Q,\mathcal{S}^3)$ with respect to $k$, hence we pass to the limit with $k\rightarrow\infty$ and we obtain

\begin{equation}
\int_{t_1}^{t_2}\int_{\Omega}\ten{T}*\eta_{\epsilon}:(\ten{\varepsilon}(\vc{u}) - \ten{\varepsilon}^{\bf p})_t*\eta_{\epsilon} \dxdt=
- 
\int_{t_1}^{t_2}\int_{\Omega}\ten{\chi}*\eta_{\epsilon}:\ten{T}^d*\eta_{\epsilon} \dxdt .
\nonumber
\end{equation}
Using the properties of convolution we get
\begin{equation}
\int_{\Omega}\ten{T}*\eta_{\epsilon}:(\ten{\varepsilon}(\vc{u}) - \ten{\varepsilon}^{\bf p})*\eta_{\epsilon} \dx \Big|_{t_1}^{t_2}=
- 
\int_{t_1}^{t_2}\int_{\Omega}\ten{\chi}*\eta_{\epsilon}:\ten{T}^d*\eta_{\epsilon}*\eta_{\delta} \dxdt .
\nonumber
\end{equation}
In the same way as in the previous section we pass to the limit with $\epsilon\rightarrow 0$ and then with $t_1\to0$
\begin{equation}
\int_{\Omega}\ten{D}(\ten{\varepsilon}(\vc{u}) - \ten{\varepsilon}^{\bf p}):(\ten{\varepsilon}(\vc{u}) - \ten{\varepsilon}^{\bf p}) \dx \Big|_{0}^{t_2}=
- 
\int_{0}^{t_2}\int_{\Omega}\ten{\chi}:\ten{T}^d \dxdt .
\label{granica_k}
\end{equation}
We multiply \eqref{granica_k} by $\frac{1}{\mu}$ and integrate over $(\tau,\tau+\mu)$ and proceed now in the same manner as in the proof of Lemma~\ref{lm:8}.% to complete the proof.
\end{proof}

The second and the third steps are conducted in the same way as in the previous limit passage, hence we omit this calculation. Using the Minty-Browder trick we show that
\begin{equation}
\ten{\chi} = \ten{G}(\tilde{\theta} + \theta,\tilde{\ten{T}}^d + \ten{T}^d )
\end{equation}
 a.e. in $Q$. Moreover using the Young measures tools we may pass to the limit in right hand side term of heat equation. Repeating the procedure from the previous limit passage we obtain
\begin{equation}
\mathcal{T}_k((\ten{T}_{k}^d + \tilde{\ten{T}}^d  ):  \ten{G}(\theta_{k} + \tilde{\theta},  \ten{T}^d_{k} + \tilde{\ten{T}}^d  ) ) 
\rightharpoonup
(\ten{T}^d + \tilde{\ten{T}}^d  ):  \ten{G}(\theta + \tilde{\theta},  \ten{T}^d + \tilde{\ten{T}}^d  ) )
\end{equation}
in $L^1(Q)$. %Now, there is no problem to pass to the limit with $k$ going to $\infty$ in \eqref{app_system_0T2a}$_{(1)}$ and \eqref{app_system_0T2po}. 
Using the solution to problem \eqref{war_brz_u} we obtain
\begin{equation}
\begin{split}
\int_0^T\int_{\Omega}(\tilde{\ten{T}}+\ten{T}):\nabla\vc{\varphi} \dxdt 
&= \int_0^T\int_{\Omega}\vc{f}\cdot \vc{\varphi} \dxdt ,
\end{split}
\label{eq:koniec1}
\end{equation}
where 
\begin{equation}
\ten{T}=\ten{D}(\ten{\varepsilon}(\vc{u}) - \ten{\varepsilon}^{\bf p}),
\end{equation}
and \eqref{eq:koniec1} holds for every test function $\vc{\varphi}\in C^{\infty}([0,T],C^{\infty}_c(\Omega,\mathbb{R}^3))$. To get the renormalised solution to heat equation let us take $S'(\theta)\phi$ as a test function in \eqref{app_system_0T2po}, where $S$ is a $C^{\infty}(\mathbb{R})$ function, such that $S'$ has a compact support. Then, by Appendix \ref{A}, limit passage in heat equation is clear and
\begin{equation}
\begin{split}
-\int_Q S(\theta)\frac{\partial \phi}{\partial t}\dxdt -&\int_{\Omega} S(\theta_0)\phi(x,0)\dx
+ \int_Q S'(\theta)\nabla\theta \cdot\nabla\phi \dxdt 
\\
+ &\int_Q S''(\theta)|\nabla(\theta)|^2\phi \dxdt
%-\int_0^T\int_{\partial\Omega}\nabla\theta\cdot\vc{n} S'(\theta- \tilde{\theta})\phi \ds\dt 
= \int_Q  \ten{G}(\theta,\ten{T}^d):\ten{T}^d S'(\theta )\phi \dxdt
\end{split}
\end{equation}
holds for every test function $\phi\in C^{\infty}_c([-\infty,T),C^{\infty}(\Omega))$ and for every function $S\in C^{\infty}(\mathbb{R})$ such that $S'\in C_0^{\infty}(\mathbb{R})$, which completes the proof of Theorem \ref{th:main}.

%%%%%%%%%%%%%%%%%%%%%%%%%%%%%%%%%%%%%%%%%%%%%%%%%%%%%%%%%%%%%%%%%%%%%%%%%%%%%%%%%%%%
%
%
%
%                       Appendix
%
%
%
%%%%%%%%%%%%%%%%%%%%%%%%%%%%%%%%%%%%%%%%%%%%%%%%%%%%%%%%%%%%%%%%%%%%%%%%%%%%%%%%%%%%

\appendix
\section{Renormalised solutions to heat equation}
\label{A}
To deal with the heat equations we introduce the renormalised solutions. Renormalised solution for parabolic equation was presented in \cite{Blanchard,BlanchardMurat}, but only for the Dirichlet boundary conditions. Some proof from \cite{Blanchard,BlanchardMurat} need a modification for the case of Neumann boundary conditions.

Let us consider the system of equations
\begin{equation}
\left\{
\begin{array}{ll}
\frac{\partial \theta^{\varepsilon}}{\partial t} -\Delta \theta^{\varepsilon} = f^{\varepsilon}
& \mbox{in }Q,
\\
\frac{\partial \theta^{\varepsilon}}{\partial \vc{n}} =0
& \mbox{on }\partial\Omega \times (0,T),
\\
\theta^{\varepsilon}(t=0)=\theta^{\varepsilon}_0,
& \mbox{in }\Omega
\end{array}
\right.
\label{eq:ren}
\end{equation}
where for every positive $\varepsilon$ the function $f^{\varepsilon}$ belongs to $L^2(Q)$ and converges weakly to $f$ in $L^1(Q)$ and $\theta^{\varepsilon}_0$ belongs to $L^2(\Omega)$ and converges strongly to $\theta_0$ in $L^1(\Omega)$ as $\varepsilon$ tends to $0$. 

In our case $\frac{1}{\varepsilon}=k$ and 
\begin{equation}
\left\{
\begin{array}{ll}
f^{\varepsilon} = \mathcal{T}_k(( \tilde{\ten{T}}^d + \ten{T}_{k}^d):  \ten{G}(\tilde{\theta} + \theta_{k},   \tilde{\ten{T}}^d + \ten{T}^d_{k}  ) ) 
& \mbox{ in } Q,
%\rightharpoonup
%(\ten{T}^d + \tilde{\ten{T}}^d  ):  \ten{G}(\theta + \tilde{\theta},  \ten{T}^d + \tilde{\ten{T}}^d  ) )
\\
\theta^{\varepsilon}(x,0) = \mathcal{T}_k(\theta_0)
& \mbox{ in } \Omega ,
\end{array}
\right.
\end{equation}
and moreover we know that sequence $\{( \tilde{\ten{T}}^d + \ten{T}_{k}^d):  \ten{G}(\tilde{\theta} + \theta_{k},   \tilde{\ten{T}}^d + \ten{T}^d_{k}  )\}$ is uniformly bounded in $L^1(Q)$. Hence, there exists a weak limit of this sequence. Identification of this weak limit is discussed in Section \ref{lim_k}.

\begin{defi}{Renormalised solution to heat equation \cite[Definition 2.2]{BlanchardMurat}}

Let $f$ belong to $L^1(Q)$ and $\theta_0$ belong to $L^1(\Omega)$. A real-valued function $\theta$ defined on $Q$ is a renormalised solution of heat equation if
\begin{itemize}
\item[a)] $\theta$ belongs to $C([0,T],L^1(\Omega))$ and $\mathcal{T}_K(\theta)$ belongs to $L^2(0,T, W^{1,2}(\Omega))$ for all positive $K$;
\item[b)] for all positive $c$ 
\begin{equation}
\mathcal{T}_{K+c}(\theta)-\mathcal{T}_k(\theta)\to 0
\end{equation}
in $L^2(0,T, W^{1,2}(\Omega))$ as $K$ goes to $\infty$;
\item[c)] and $\theta(t=0)=\theta_0$.
\end{itemize}
Moreover, for all functions $S\in C^{\infty}(\mathbb{R})$, such that $S'$ belongs to $C^{\infty}_0(\mathbb{R})$ ($S'$ has a compact support), the following equality holds
\begin{equation}
\begin{split}
-\int_Q S(\theta)\frac{\partial \phi}{\partial t}\dxdt -\int_{\Omega} S(\theta_0)\phi(x,0)\dx
+ &\int_Q S'(\theta)\nabla\theta\cdot\nabla\phi \dxdt 
\\
+ &\int_Q S''(\theta)|\nabla\theta|^2\phi \dxdt
%-\int_0^T\int_{\partial\Omega}\nabla\theta\cdot\vc{n} S'(\theta)\phi \ds\dt % BO ZEROWE WARUNKI BRZEGOWE
= \int_Q  fS'(\theta)\phi \dxdt
\end{split}
\end{equation}
for all $\phi\in C_0^{\infty}(Q)$.
\end{defi}
%The following definition came form \cite{BlanchardMurat}.

We use the notation $\lim_{\eta,\varepsilon\to 0}$ when the order in the passing to the limit is not relevant, i.e.
\begin{equation}
\lim_{\eta,\varepsilon\to 0} F_{\eta,\varepsilon} =
\lim_{\eta\to 0}\lim_{\varepsilon\to 0} F_{\eta,\varepsilon} =
\lim_{\varepsilon\to 0}\lim_{\eta\to 0} F_{\eta,\varepsilon}.
\nonumber
\end{equation}

\begin{lemat}
There exists a subsequence of the sequence $\{\theep\}_{\varepsilon}$ (still denoted by $\varepsilon$) and $\theta \in C([0,T],L^1(\Omega))$, such that when $\varepsilon$ tends to $0$ and for any fixed positive real number $K$ the following conditions are satisfied 
\begin{itemize}
\item[a)] $\theep$ converges almost everywhere in $Q$ to a measurable function $\theta$;
\item[b)] $\theep$ converges to $\theta$ in $C([0,T],L^1(\Omega))$;
\item[c)] $\mathcal{T}_K(\theep)$ converges weakly to $\mathcal{T}_K(\theta)$ in $L^2(0,T,W^{1,2}(\Omega))$;
\item[d)] there exists the following limit
\begin{equation}
\lim_{\eta,\varepsilon\to \infty} \int_Q |\nabla \mathcal{T}_k(\theep-\theet)|=0.
\end{equation}
\end{itemize}
\end{lemat}

\begin{proof}
Let us take $\mathcal{T}_K(\theep)$ as a test function in \eqref{eq:ren}. Then for $t\in (0,T)$
\begin{equation}
\int_0^t\int_{\Omega}\frac{\partial \theep}{\partial t} \mathcal{T}_K(\theep) \dxdt + 
\int_0^t\int_{\Omega}|\nabla \mathcal{T}_K(\theep)|^2\dxdt =
\int_0^t\int_{\Omega}f^{\varepsilon}\mathcal{T}_k(\theep),
\end{equation}
and
\begin{equation}
\int_{\Omega}\tilde{\mathcal{T}}_K(\theep)(t) \dx + 
\int_0^t\int_{\Omega}|\nabla \mathcal{T}_K(\theep)|^2\dxdt =
\int_0^t\int_{\Omega}f^{\varepsilon}\mathcal{T}_k(\theep) +
\int_{\Omega}\tilde{\mathcal{T}}_K(\theep_0) \dx ,
\end{equation}
where $\tilde{\mathcal{T}}_K(r)=\int_0^r \mathcal{T}_K(z)\dz$ is a positive real valued function. Using definition of the truncation and linear growth of function $\tilde{\mathcal{T}}_K(r)$ at infinity, the following estimate holds
\begin{equation}
\int_{\Omega}\tilde{\mathcal{T}}_K(\theep)(t) \dx + 
\int_0^t\int_{\Omega}|\nabla \mathcal{T}_K(\theep)|^2\dxdt \leq
K \|f\|_{L^1(Q)} + C(K)\|\theep_0\|_{L^1(\Omega)}.
\end{equation}
%Using triangle and Poincar\'{e} inequality we obtain
%\begin{equation}
%\begin{split}
%\|\mathcal{T}_K(\theep)\|_{L^2(Q)} &\leq \|T_K(\theep) - (\mathcal{T}_K(\theep))_{\Omega}\|_{L^2(Q)} + \|(\mathcal{T}_K(\theep))_{\Omega}\|_{L^2(Q)}
%\\
%&\leq \|\nabla \mathcal{T}_K(\theep) \|_{L^2(Q)} + \|(\mathcal{T}_K(\theep))_{\Omega}\|_{L^2(Q)},
%\end{split}
%\end{equation}
%where by $(\mathcal{T}_K(\theep))_{\Omega}$ we denote  the mean value.
To show that the sequence $\{\mathcal{T}_K(\theep)\}_{\varepsilon>0}$ is uniformly bounded in $L^2(0,T,W^{1,2}(\Omega))$, it is enough to estimate $\|\mathcal{T}_K(\theep)\|_{L^2(Q)}$ by $\|\tilde{\mathcal{T}}_K(\theep)\|_{L^1(Q)}$ and $\|\nabla \mathcal{T}_K(\theep)\|_{L^2(Q)}$. By Poincar\'{e} inequality we get 
\begin{equation}
\begin{split}
\|\mathcal{T}_K(\theep)\|_{L^2(Q)} &\leq \|\mathcal{T}_K(\theep) - (\mathcal{T}_K(\theep))_{\Omega}\|_{L^2(Q)} + \|(\mathcal{T}_K(\theep))_{\Omega}\|_{L^2(Q)}
\\
&\leq \|\nabla \mathcal{T}_K(\theep) \|_{L^2(Q)} + \|(\mathcal{T}_K(\theep))_{\Omega}\|_{L^2(Q)},
\end{split}
\end{equation}
where by $(\mathcal{T}_K(\theep))_{\Omega}$ we denote the mean value. Using the definition of truncation operator we obtain
\begin{equation}
\tilde{\mathcal{T}}_K(\theep)=
\left\{
\begin{array}{ll}
\frac{1}{2}(\theep)^2 & |\theep|\leq K,
\\
\frac{1}{2}K^2 + K(\theep -K) & |\theep|> K,
\end{array}
\right.
\label{eq:tilteT}
\end{equation}
and then it remains to show the estimates for $(\mathcal{T}_K(\theep))_{\Omega}$
%See picture \ref{pic:}. Using \eqref{eq:tilteT} we obtain 
\begin{equation}
\int_{\Omega} |\mathcal{T}_K(\theep)|^2 \dx = 
\int_{\{x\in\Omega : |\theep|\leq K\} } |\theep|^2 +
\int_{\{x\in\Omega : |\theep|> K\} } K^2
\leq
%\int_{\Omega}|\mathcal{T}_K(\theep)|^2 \dx \leq 
2\int_{\Omega}\tilde{\mathcal{T}}_K(\theep)\dx .
\end{equation}
The finite measure of $Q$ implies that the sequence $\{\mathcal{T}_K(\theep)\}_{\varepsilon >0}$ is uniformly bounded in $L^2(0,T,W^{1,2}(\Omega))$.

For $\delta>0$, let us test the difference of two approximate equations \eqref{eq:ren}  
\begin{equation}
\frac{\partial}{\partial t}(\theep - \theet) - \Delta (\theep - \theet)= f^{\varepsilon}- f^{\eta}.
\label{eq:roznica}
\end{equation}
by function $\frac{1}{\delta}\mathcal{T}_{\delta}(\theep-\theet)$. As a result, we get
\begin{equation}
\frac{1}{\delta}\int_{\Omega}\tilde{\mathcal{T}}_{\delta}(\theep - \theet)(t) + %\underbrace{
\frac{1}{\delta}\int_0^t\int_{\Omega}|\nabla \mathcal{T}_{\delta} (\theep - \theet)|^2%}_{\geq 0} 
= \frac{1}{\delta}\int_0^t\int_{\Omega}(f^{\varepsilon} - f^{\eta})\mathcal{T}_{\delta}(\theep - \theet) + \frac{1}{\delta}\int_{\Omega} \tilde{\mathcal{T}}_{\delta}(\theep_0 - \theet_0).
\end{equation}
Using the positivity of the second term of left hand side we obtain
\begin{equation}
\frac{1}{\delta}\int_{\Omega}\tilde{\mathcal{T}}_{\delta}(\theep - \theet)(t) \leq
 \int_0^t\int_{\Omega}|f^{\varepsilon} - f^{\eta}| + \frac{1}{\delta}\int_{\Omega} \tilde{\mathcal{T}}_{\delta}(\theep_0 - \theet_0).
\end{equation}
Passing to the limit as $\delta$ which tends to $0$ we obtain
\begin{equation}
\begin{split}
\lim_{\delta\to 0}\frac{1}{\delta}\int_{\Omega}\tilde{\mathcal{T}}_{\delta}(\theep - \theet)(t) &=
\int_{\Omega}(\theep - \theet)(t) \dx ,
\\
\lim_{\delta\to 0}\frac{1}{\delta}\int_{\Omega}\tilde{\mathcal{T}}_{\delta}(\theep_0 - \theet_0) &=
\int_{\Omega}(\theep_0 - \theet_0) \dx.
\end{split}
\end{equation}
Therefore,
\begin{equation}
\int_{\Omega}(\theep - \theet)(t) \dx
\leq
\int_0^t\int_{\Omega}|f^{\varepsilon} - f^{\eta}|
+
\int_{\Omega}(\theep_0 - \theet_0) \dx ,
\end{equation}
and we conclude that the sequence $\{\theep\}$ is a Cauchy sequence in $C([0,T],L^1(\Omega))$, hence there exists $\theta \in C([0,T],L^1(\Omega))$, such that $\theep \to \theta$ in $C([0,T],L^1(\Omega))$ as $\varepsilon$ tends to $0$.  %{\color{red} I tu chyba mamy zbieżność punktową prawie wszędzie  }

Testing the equation \eqref{eq:roznica} by $\mathcal{T}_K(\theep-\theet)$, we obtain
\begin{equation}
\int_{\Omega}\tilde{\mathcal{T}}_K(\theep - \theet)(T) + \int_Q |\nabla \mathcal{T}_K(\theep-\theet)|^2 =
\int_Q (f^{\varepsilon} - f^{\eta})\mathcal{T}_K(\theep - \theet) + \int_{\Omega} \tilde{\mathcal{T}}_K(\theep_0 - \theet_0)
\end{equation}
Positivity of the first term on the left hand side in abovementioned equation and the convergences of right hand side functions and initial conditions imply that
\begin{equation}
\lim_{\varepsilon,\eta\to 0}\int_Q |\nabla \mathcal{T}_K(\theep-\theet)|^2 = 0
\end{equation}
which completes the proof.
\end{proof}

\begin{lemat}
Let $K$ be a fixed positive real number. The sequence $\{\mathcal{T}_K(\theep)\}$ converges strongly to $\mathcal{T}_K(\theta)$ in $L^2(0,T,W^{1,2}(\Omega))$.
\label{lm:mocna_zbieznosc}
\end{lemat}
The proof of this lemma can be found in \cite{Blanchard}.

Multiplying \eqref{eq:ren} by $S'(\theta^{\varepsilon})\phi$, where $S\in C^{\infty}(\mathbb{R})$ and $S'$ has a compact support and $\phi\in C_0^{\infty}(Q)$, we get
\begin{equation}
\begin{split}
-\int_Q S(\theta^{\varepsilon})\frac{\partial \phi}{\partial t}\dxdt -&\int_{\Omega} S(\theta^{\varepsilon}_0)\phi(x,0)\dx
+ \int_Q S'(\theta^{\varepsilon})\nabla\theta^{\varepsilon}\cdot\nabla\phi \dxdt 
\\
+ &\int_Q S''(\theta^{\varepsilon})|\nabla\theta^{\varepsilon}|^2\phi \dxdt
= \int_Q  f^{\varepsilon}S'(\theta^{\varepsilon})\phi \dxdt .
\end{split}
\label{eq:Sprim}
\end{equation}
$S'$ has a compact support, hence there exist $0<M<\infty$ such that $\mbox{supp}(S')\subset [-M,M]$. This allows us to enter into  equation \eqref{eq:Sprim} the truncations operator
\begin{equation}
\begin{split}
-\int_Q S(\theta^{\varepsilon})\frac{\partial \phi}{\partial t}\dxdt -&\int_{\Omega} S(\theta^{\varepsilon}_0)\phi(x,0)\dx
+ \int_Q S'(\mathcal{T}_M(\theta^{\varepsilon}))
\nabla\mathcal{T}_M(\theta^{\varepsilon}) \cdot\nabla\phi \dxdt 
\\
+ &\int_Q S''(\mathcal{T}_M(\theta^{\varepsilon}))|\nabla\mathcal{T}_M(\theta^{\varepsilon})|^2\phi \dxdt
= \int_Q  f^{\varepsilon}S'(\mathcal{T}_M(\theta^{\varepsilon}))\phi \dxdt .
\end{split}
\label{eq:Sprim2}
\end{equation}
Using the Egorov theorem applied to $S'(\theta^{\varepsilon})$ or to $S''(\theta^{\varepsilon})$ and using the bounded character of the remaining terms we can pass to the limit with $\varepsilon$ going to 0 in \eqref{eq:Sprim2} and we obtain
\begin{equation}
\begin{split}
-\int_Q S(\theta)\frac{\partial \phi}{\partial t}\dxdt -&\int_{\Omega} S(\theta_0)\phi(x,0)\dx
+ \int_Q S'(\mathcal{T}_M(\theta))
\nabla\mathcal{T}_M(\theta) \cdot\nabla\phi \dxdt 
\\
+ &\int_Q S''(\mathcal{T}_M(\theta))|\nabla\mathcal{T}_M(\theta)|^2\phi \dxdt
= \int_Q  fS'(\mathcal{T}_M(\theta))\phi \dxdt .
\end{split}
\label{eq:Sprim3}
\end{equation}
And finally, using the compact support of $S'$ we can omit the truncations in \eqref{eq:Sprim3}
\begin{equation}
\begin{split}
-\int_Q S(\theta)\frac{\partial \phi}{\partial t}\dxdt -&\int_{\Omega} S(\theta_0)\phi(x,0)\dx
+ \int_Q S'(\theta)
\nabla\theta \cdot\nabla\phi \dxdt 
\\
+ &\int_Q S''(\theta)|\nabla\theta|^2\phi \dxdt
= \int_Q  fS'(\theta)\phi \dxdt ,
\end{split}
\label{eq:Sprim4}
\end{equation}
which completes the proof of existence regarding renormalised solution to parabolic equation with Neumann boundary condition.

\begin{lemat}
Assuming that $\theta_{0,1}$ and $\theta_{0,2}$ lie in $L^1(\Omega)$, $f_{1}$ and $f_{2}$ lie in $L^1(Q)$ and they satisfy 
\begin{equation}
\left\{
\begin{split}
\theta_{0,1} &\leq\theta_{0,2}
\\
f_{1} &\leq f_{2}
\end{split}
\right.
\end{equation}
Then if $\theta_{1}$ and $\theta_{2}$ are two renormalised solutions respectively for date $(\theta_{0,1},f_1)$ and $(\theta_{0,2},f_2)$, we have 
\begin{equation}
\theta_{1}\leq\theta_{2}
\end{equation}
almost everywhere in $Q$.
\label{lm:uniq}
\end{lemat}
Proof of this lemma can be found in \cite{BlanchardMurat}.

\begin{uwaga}
As a consequence of Lemma \ref{lm:uniq}, the renormalised solution is unique.
\end{uwaga}

\section{Orlicz spaces tools}
\label{B}
Assumption \ref{ass_G} requires the use of basic tools regarding generalized Orlicz spaces. Here we present some basic lemmas, which have been used to prove the existence of thermo-visco-elastic model solution. Following lemmas with the proof can be found in \cite{GSW,GWWZ,G1,EW}.
%\fxnote{Missing cite}

\begin{lemat}{Fenchel-Young inequality}

Let $M$ be an $N$-function and $M^*$ be complementary to $M$. Then the following inequality is satisfied
\begin{equation}
|\ten{\xi}:\ten{\eta}|\leq M(x,\ten{\xi}) + M^*(x,\ten{\eta})
\end{equation}
for all $\ten{\xi},\ten{\eta}\in \mathcal{S}^3$ and for almost all $x\in\Omega$.
\label{lm:funchel-Young}
\end{lemat}
%Proof of Funchel-Young inequality came from \cite{}.
%\begin{proof}
%
%\end{proof}

\begin{lemat}{H\"{o}lder inequality}

Let $M$ be an $N$-function and $M^*$ be complementary to $M$. Then the following inequality is satisfied
\begin{equation}
|\int_Q \ten{\xi}:\ten{\eta}\dxdt|\leq 2\|\ten{\xi}\|_{L,M}\| \ten{\eta}\|_{L,M^*}.
\end{equation}
%for all $\ten{\xi},\ten{\eta}\in \mathcal{S}^3$ and for almost all $x\in\Omega$.
\label{lm:Cauchy}
\end{lemat}

%%%\begin{lemat}
%%%Let $M$ be an $N$-function and $M^*$ its complementary, then
%%%\begin{equation}
%%%|\intO \ten{\xi}:\ten{\eta}\dx | \leq 2\|\ten{\xi}\|_M\|\ten{\eta}\|_{M^*},
%%%\end{equation}
%%%where $\ten{\xi}\in L_M(\Omega,\mathcal{S}^3)$ and $\ten{\eta}\in L_{M^*}(\Omega,\mathcal{S}^3)$.
%%%\end{lemat}
%%%
%%%
%%%
%%%
%%%\begin{lemat}
%%%If $M$ satisfies $\Delta_2$-condition, see \eqref{eq:def_delta2}, then $\mathcal{L}_M(Q)$ is a vector space.
%%%\end{lemat}
%%%\fxnote{Sformuować dokładnie}
%\begin{proof}
%
%\end{proof}

%{\color{red} Next we recall an analogue to the Vitali's lemma, however for the modular convergence instead of the strong convergence in $L^p$.  }

%\fxnote{Do wywalenia???}
\begin{lemat}\label{lm:16}
Let $\ten{\xi}_i:Q\to \mathbb{R}^d$ be a measurable sequence. Then $\ten{\xi}_i \xrightarrow{M} \ten{\xi}$ in $L_M(Q)$ modularity if and only if $\ten{\xi}_i \to \ten{\xi}$ in measure and there exist some $\lambda>0$ such that the sequence $\{M(\cdot,\lambda \ten{\xi}_i)\}$ is uniformly integrable, i.e.
\begin{equation}
\lim_{R\to\infty} \left(\sup_{i\in\mathbb{N}} \int_{ \{(t,x):\ |M(x,\lambda\ten{\xi}_i)|\geq R\} } M(x,\lambda\ten{\xi}_i)\dxdt\right) =0 .
\end{equation}
\end{lemat}

%\begin{proof}
%
%\end{proof}

\begin{lemat}\label{lm:17}
Let $M$ be an $N$-function and for all $i\in\mathbb{N}$, let $\int_Q M(x,\ten{\xi}_i) \dxdt \leq c$. Then the sequence $\{\ten{\xi}_i\}$ is uniformly integrable.
\end{lemat}
%
%\begin{proof}
%
%\end{proof}

\begin{lemat}
Let $M$ be an $N$-function and $M^*$ its complementary function. Suppose that the sequences $\ten{\Phi}_i: Q\to\mathcal{S}^3$ and $\ten{\Psi}_i: Q\to\mathcal{S}^3$ are uniformly bounded in $L_M(Q)$ and $L_{M^*}(Q)$, respectively. Moreover, $\ten{\Phi}_i \xrightarrow{M} \ten{\Phi}$ modularly in $L_M(Q)$ and $\ten{\Phi}_i \xrightarrow{M^*} \ten{\Phi}$ modularly in $L_{M^*}(Q)$. Then, $\ten{\Phi}_i:\ten{\Psi}_i \to \ten{\Phi}:\ten{\Psi}$ strongly in $L^1$.
\label{lm:18}
\end{lemat}

%\begin{proof}
%
%\end{proof}

\begin{lemat}\label{lm:19}
Let $\rho_i$ be a standard mollifier, i.e. $\rho\in C^{\infty}(\mathbb{R})$, $\rho$ has a compact support and $\int_{\mathbb{R}} \rho(\tau)\dtau =1$, $\rho(\tau)=\rho(-\tau)$. We define $\rho_i(\tau)= i\rho(i\tau)$. Moreover, let $*$ denote a convolution in the variable $\tau$. Then for any function $\ten{\Phi}:Q\to\mathcal{S}^3$, such that $\ten{\Phi}\in L^1(Q,\mathcal{S}^3$, it holds
\begin{equation}
\rho_i*\ten{\Phi} \to \ten{\Phi} \qquad \mbox{in measure.}
\end{equation}
\end{lemat}

%\begin{proof}
%
%\end{proof}

\begin{lemat}\label{lm:20}
Let $\rho_i$ be a standard mollifier. Given an $N$-function $M$ and a function $\ten{\Phi}:Q\to\mathcal{S}^3$ such that $\ten{\Phi}\in\mathcal{L}_M(Q)$, the sequence $\{M(x,\rho_i*\ten{\Phi})\}$ is uniformly integrable.
\end{lemat}

%
%\begin{proof}
%
%\end{proof}

\section{Young measures tools}\label{C}
Right hand side term in the approximated heat equation is a product of elements of two sequences which converge weakly. To characterize the limit of this term we use Young measure theory. In this section, we present necessary lemmas. They come from \cite[Corollaries 3.2-3.4]{muller}. Similar technique was also used in \cite{PGAS,14Agnieszka}. 

%{\color{red}See also \cite{Ball}}.

\begin{lemat} 
Suppose that the sequence of maps $z_j: Q \to \mathbb{R}^d$ generates the Young measure $\nu: Q \to \mathcal{M}(\mathbb{R}^d)$. Let $F : \Omega \times \mathbb{R}^d \to \mathbb{R}^d$ be a Carath\'eodory function (i.e. measurable in the first argument and continuous in the second). Let us also assume that the negative part $F^- (x, z_j(x,t))$ is weakly relatively compact in $L^1(Q)$. Then
\begin{equation}
\liminf_{j\to\infty}\int_E F(x,z_j(x,t))\dxdt \geq 
\int_E\int_{\mathbb{R}^d} F(x,\lambda)\dnu_x(\lambda) \dxdt
\end{equation}
If, in addition, the sequence of functions $x \to |F|(x, z_j (x,t))$ is weakly relatively compact in $L^1(Q)$, then
\begin{equation}
F (\cdot, z_j (\cdot,\cdot)) \rightharpoonup \int_{\mathbb{R}^d} F (\cdot, \lambda) \dnu_x(\lambda) \qquad \mbox{in } L^1(Q).
\end{equation} 
\label{lm:ineq}
\end{lemat}

\begin{lemat} 
Let $u_j : Q \to \mathbb{R}^d$ , $v_j : Q \to \mathbb{R}^{d'}$ be measurable and suppose that $u_j \to u$ a.e. while $v_j$
generates the Young measure $\nu$. Then the sequence of pairs $(u_j , v_j ) : Q \to \mathbb{R}^{d+d'}$ generates the Young measure $x \to \delta_{u(x)}\otimes \nu_x$.
\label{lm:charakteryzacja}
\end{lemat}
\begin{lemat}
Suppose that a sequence $z_j$ of measurable functions from $Q$ to $\mathbb{R}^d$ generates the Young measure $\nu : Q \to \mathcal{M}(\mathbb{R}^d)$. Then $z_j\to z$ in measure if and only if $\nu_x = \delta_{z(x)}$ a.e..
\label{lm:delta_measure}
\end{lemat}

%\listoffixmes

{\bf Acknowledgement} F.K. is a PhD student of the International PhD Projects Programme of Foundation for Polish Science operated within the Innovative Economy Operational Programme 2007-2013 funded by European Regional Development Fund (PhD Programme: Mathematical Methods in Natural Sciences). The project was financed by the National Science Centre, the number of decisions DEC-2012/05/E/ST1/02218

 %Partially supported by grant NCN OPUS 2012/07/B/ST1/03306.

%\bibliography{thermoelasticity_orlicz.bib}{}
\bibliographystyle{plain}

\end{document}